\documentclass[11pt, reqno]{amsart}

\usepackage{amsmath}

\usepackage{tikz}
\usepackage[margin=1.25in]{geometry}
\usepackage[matrix,arrow,curve,frame]{xy}
\usepackage{hyperref}
\usepackage{amsthm}
\usepackage{amsfonts}
\usepackage{amssymb}

\definecolor{my-linkcolor}{rgb}{0.75,0,0}
\definecolor{my-citecolor}{rgb}{0.1,0.57,0}
\definecolor{my-urlcolor}{rgb}{0,0,0.75}
\hypersetup{
	colorlinks, 
	linkcolor={my-linkcolor},
	citecolor={my-citecolor}, 
	urlcolor={my-urlcolor}
}

\title[Finite-index vertex operator subalgebras]{On semisimplicity of module categories for finite non-zero index vertex operator subalgebras}
\author{Robert McRae}
\date{}
\address{Yau Mathematical Sciences Center, Tsinghua University, Beijing 100084, China}
\curraddr{}
\email{rhmcrae@tsinghua.edu.cn}

\theoremstyle{definition}\newtheorem{rema}{Remark}[section]
\theoremstyle{plain}
\newtheorem{theo}[rema]{Theorem}

\theoremstyle{definition}\newtheorem{defi}[rema]{Definition}
\theoremstyle{plain}\newtheorem{lemma}[rema]{Lemma}
\newtheorem{corol}[rema]{Corollary}
\theoremstyle{definition}\newtheorem{exam}[rema]{Example}
\theoremstyle{definition}
\theoremstyle{definition}
\theoremstyle{definition}\newtheorem{assum}[rema]{Assumption}

\newcommand{\cA}{\mathcal{A}}
\newcommand{\cR}{\mathcal{R}}

\newcommand{\cS}{\mathcal{S}}
\newcommand{\cF}{\mathcal{F}}

\newcommand{\cC}{\mathcal{C}}

\newcommand{\cW}{\mathcal{W}}

\newcommand{\CC}{\mathbb{C}}
\newcommand{\ZZ}{\mathbb{Z}}
\newcommand{\NN}{\mathbb{N}}

\newcommand{\Id}{\mathrm{Id}}

\newcommand{\tens}{\boxtimes}
\newcommand{\vac}{\mathbf{1}}
\DeclareMathOperator{\im}{Im}
\DeclareMathOperator{\coker}{Coker}

\let\ker\relax

\DeclareMathOperator{\ker}{Ker}

\DeclareMathOperator{\Endo}{End}

\begin{document}
\bibliographystyle{alpha}

\numberwithin{equation}{section}

\begin{abstract}
 Let $V\subseteq A$ be a conformal inclusion of vertex operator algebras and let $\cC$ be a category of grading-restricted generalized $V$-modules that admits the vertex algebraic braided tensor category structure of Huang-Lepowsky-Zhang. We give conditions under which $\cC$ inherits semisimplicity from the category of grading-restricted generalized $A$-modules in $\cC$, and vice versa. The most important condition is that $A$ be a rigid $V$-module in $\cC$ with non-zero categorical dimension, that is, we assume the index of $V$ as a subalgebra of $A$ is finite and non-zero. As a consequence, we show that if $A$ is strongly rational, then $V$ is also strongly rational under the following conditions: $A$ contains $V$ as a $V$-module direct summand, $V$ is $C_2$-cofinite with a rigid tensor category of modules, and $A$ has non-zero categorical dimension as a $V$-module. These results are vertex operator algebra interpretations of theorems proved for general commutative algebras in braided tensor categories. We also generalize these results to the case that $A$ is a vertex operator superalgebra.
\end{abstract}

\maketitle

 \tableofcontents

\section{Introduction}

Vertex operator algebras, which are an approach to the rigorous mathematical construction of two-dimensional conformal quantum field theories, are rather complicated algebraic structures and can be difficult to construct from scratch. As a result, some of the best sources of new vertex operator algebras are vertex operator subalgebras of old ones, and extensions, where a new vertex operator algebra is constructed on a sum of modules for an old one. For example, the celebrated moonshine module vertex operator algebra \cite{FLM} was first constructed as an extension of a subalgebra of the Leech lattice vertex operator algebra. However, while subalgebras and extensions are rich sources of new vertex operator algebras, it is not always clear how ``nice'' the new algebra is, compared to the old one.

In this paper, we consider whether the representation category of a vertex operator subalgebra or extension inherits semisimplicity from that of the original algebra. In the presence of the technical $C_2$-cofiniteness condition, semisimplicity of the full (grading-restricted generalized) module category is equivalent to the condition of rationality originally introduced by Zhu \cite{Zh}, which may be viewed as a mathematical characterization of the vertex operator algebras appearing in rational conformal field theory. More recently, the term ``strongly rational'' (or ``strongly regular'') has been applied to vertex operator algebras which are simple, self-contragredient, positive energy (also known as CFT-type), $C_2$-cofinite, and rational. The full module category of such a vertex operator algebra is a (semisimple) modular tensor category \cite{Hu-rig_mod}. 

A long-standing question asks when subalgebras of strongly rational vertex operator algebras are strongly rational. For example, the ``orbifold rationality problem'' asks if the fixed-point subalgebra of a finite group of automorphisms of a strongly rational vertex operator algebra is strongly rational; this problem has recently been solved for solvable automorphism groups \cite{CM} but remains open for general finite groups. Similarly, the ``coset rationality problem'' asks whether a tensor product subalgebra $U\otimes V$ of a strongly rational vertex operator algebra is strongly rational, provided $U$ and $V$ are mutual commutants and one of them is strongly rational. Such rationality problems are a major motivation for this paper, and as an application of our semisimplicity results, we present some sufficient conditions for a subalgebra of a strongly rational vertex operator algebra to be strongly rational (see Theorem \ref{thm:voa_rat_thm_intro} below).

Our main theorem in this work is based on the deep result, most recently developed by Huang, Lepowsky, and Zhang \cite{HLZ1}-\cite{HLZ8}, that suitable module categories for vertex operator algebras are braided tensor categories,  meaning that we can apply tools from tensor category theory to study vertex operator subalgebras and extensions. Suppose we have an inclusion $V\subseteq A$ of vertex operator algebras (with the same conformal vector). If $\cC$ is a category of (grading-restricted generalized) $V$-modules that includes $A$ and has the braided tensor category structure described in \cite{HLZ8}, then \cite[Theorem 3.2]{HKL} shows that $A$ is a commutative algebra (defined in Section \ref{sec:CA_ss_to_C_ss} below) in the braided tensor category $\cC$. In particular, there is a multiplication homomorphism $\mu_A: A\tens A\rightarrow A$ induced by the vertex operator on $A$ and a unit homomorphism $\iota_A: V\rightarrow A$ (which is simply the inclusion). By \cite[Theorem 3.4]{HKL}, the category $\cC_A^0$ of ``local'' or ``dyslectic'' modules for the commutative algebra $A$ in $\cC$ agrees with the category of (grading-restricted generalized) $A$-modules which, viewed as $V$-modules, are objects of $\cC$. Now here is our main theorem; see Theorem \ref{thm:main_vosa_thm} below for the generalization to the case that $A$ is a vertex operator superalgebra: 
\begin{theo}\label{thm:main_voa_thm_intro}
 Suppose $V\subseteq A$ is a conformal inclusion of vertex operator algebras and $\cC$ is a category of grading-restricted generalized $V$-modules that includes $A$ and admits the braided tensor category structure of \cite{HLZ8}. Assume that moreover:
 \begin{itemize}
  \item There is a $V$-homomorphism $\varepsilon_A: A\rightarrow V$ such that $\varepsilon_A\circ\iota_A=\Id_V$.
  \item The vertex operator algebra $A$ is a rigid and self-dual object of $\cC$ with evaluation $\varepsilon_A\circ\mu_A: A\tens A\rightarrow V$ and some coevaluation $i_A: V\rightarrow A\tens A$.
  \item For some non-zero index $[A:V]\in\CC$, $\mu_A\circ i_A=[A:V]\iota_A$  .
 \end{itemize}
Then:
\begin{enumerate}
 \item If $\cC$ is semisimple, then $\cC_A^0$ is also semisimple.
 \item If $\cC$ is rigid and $\cC_A^0$ is semisimple, then $\cC$ is also semisimple.
\end{enumerate}
\end{theo}

Only the second assertion of this theorem assumes rigidity for the full tensor category $\cC$ of $V$-modules. In this sense, semisimplicity for a subalgebra is more difficult than for an extension. However, both semisimplicity results require non-vanishing of the index $[A:V]$, which turns out to be the categorical dimension of $A$ as a $V$-module assuming $A$ is a $\ZZ$-graded vertex operator algebra. (The assumption that $A$ is rigid in a tensor category of $V$-modules means that this index is finite, that is, $A$ is ``finite-dimensional as a $V$-module,'' equivalently, $V$ is a ``finite-index subalgebra of $A$.'') So to prove, for example, that $A$ inherits strong rationality from $V$, we must first check that $[A:V]\neq 0$. As we will show in Section \ref{sec:C_ss_to_CA_ss}, Theorem \ref{thm:main_voa_thm_intro}(1) is a vertex operator algebra version of Maschke's Theorem for finite groups, and the non-vanishing index requirement corresponds to the (essential!) requirement in Maschke's Theorem that the order of the group be non-zero in the field.

Both assertions of the main theorem are vertex operator algebra interpretations (via \cite{HKL}) of theorems for (commutative) algebras in (braided) tensor categories. Indeed, Theorem \ref{thm:main_voa_thm_intro}(1) is not original to this paper since it is an almost immediate consequence of \cite[Theorem 3.3]{KO}. We have included the result here to compare with the converse result Theorem \ref{thm:main_voa_thm_intro}(2), to emphasize the importance of the non-vanishing index condition, and to present a proof that directly generalizes standard proofs of Maschke's Theorem (unlike the proof in \cite{KO}). Thus in our proof, the non-vanishing index condition plays the same role as the corresponding condition in Maschke's Theorem: it is needed for ``averaging'' a $V$-module homomorphism into an $A$-module homomorphism.

The proof of Theorem \ref{thm:main_voa_thm_intro}(2) generalizes part of the argument used in \cite{CM} to prove strong rationality for orbifold subalgebras associated to finite cyclic automorphism groups. The basic idea is to relate the representation theory of $V$ to the semisimple representation theory of $A$ using the induction functor of \cite{KO, CKM1}. But usually, $V$-modules induce to ``non-local'' $A$-modules which are not objects of the semisimple braided tensor category $\cC_A^0$. Thus after inducing a $V$-module, we have to project the resulting possibly non-local $A$-module to $\cC_A^0$ using a functor from \cite{KO, FFRS, McR}; the construction of this projection functor is where we use the assumption $[A:V]\neq 0$. Induction and projection applied to $V$ itself yields a \textit{non-zero} object in $\cC_A^0$, namely $A$, allowing us to use the semisimplicity of $\cC_A^0$ to prove $V$ is projective in $\cC$. Since $V$ is the unit object of $\cC$ and we assume $\cC$ is rigid, this is enough to conclude $\cC$ is semisimple (see for example \cite[Corollary 4.2.13]{EGNO}).

In Theorem \ref{thm:main_voa_thm_intro}(2), we would not need to assume that $\cC$ is rigid and that $[A:V]\neq 0$ if we assumed instead that the full category of non-local $A$-modules in $\cC$ is semisimple. But non-local modules for vertex operator algebras are not well understood in general, so there does not seem to be any general practical way to show they are semisimple. For orbifold extensions, that is, $V$ is the fixed-point subalgebra of an automorphism group of $A$, non-local $A$-modules include twisted $A$-modules associated to automorphisms of $A$. But even in this particular setting, semisimplicity for twisted modules is not clear even if we assume untwisted $A$-modules are semisimple. Thus the point of Theorem \ref{thm:main_voa_thm_intro}(2) is to provide criteria for $\cC$ to be semisimple that do not require any knowledge about non-local $A$-modules. 

As an application of Theorem \ref{thm:main_voa_thm_intro}, we prove the following rationality theorem; see Theorem \ref{thm:vosa_rat_1} for the generalization to the case that $A$ is a vertex operator superalgebra:
\begin{theo}\label{thm:voa_rat_thm_intro}
 Let $V\subseteq A$ be a conformal inclusion of vertex operator algebras.
 \begin{enumerate}
  \item Assume $V$ is strongly rational. Then $A$ is also strongly rational if and only if:
  \begin{itemize}
   \item $A$ is simple and positive energy.
   \item The dimension of $A$ in the modular tensor category of $V$-modules is non-zero.
  \end{itemize}
  
  \item Assume $A$ is strongly rational. Then $V$ is also strongly rational if and only if:
\begin{itemize}
 \item There is a $V$-module homomorphism $\varepsilon_A: A\rightarrow V$ such that $\varepsilon_A\vert_V=\Id_V$.
  \item $V$ is $C_2$-cofinite.
  \item The tensor category $\cC$ of grading-restricted generalized $V$-modules is rigid.
  \item The categorical dimension $\dim_\cC A$ is non-zero.
 \end{itemize}
 \end{enumerate}
\end{theo}

If $V$ is positive energy and $C_2$-cofinite, then its full module category $\cC$ is already a braided tensor category \cite{Hu}, so the third condition in Theorem \ref{thm:voa_rat_thm_intro}(2) is simply that $\cC$ is rigid. Rigidity of $\cC$ for positive-energy self-contragredient $C_2$-cofinite $V$ is widely believed to be automatic   (see for example \cite[Conjecture 4.2]{Hu-C2-conj}), but so far this is known only in a few examples. We will discuss whether other conditions in Theorem \ref{thm:voa_rat_thm_intro} may be automatic or vacuous in Section \ref{sec:VOA_app}. Perhaps the most interesting question is whether there exists a simple positive-energy vertex operator algebra which is an extension of a strongly rational subalgebra and has categorical dimension $0$. More generally, does there exist a commutative algebra with dimension $0$ and trivial twist in a $\CC$-linear semisimple modular tensor category? Although it seems likely that such algebras exist, the author does not know any examples. 

Now to what specific examples of vertex operator subalgebras and extensions does Theorem \ref{thm:voa_rat_thm_intro} apply? Non-vanishing of the index has been verified for extensions of orbifold type \cite{McR1} and of coset type \cite{CKM}. In Example \ref{exam:orbifold} below, we will discuss the implications of the theorem for orbifold subalgebras and extensions, although we do not get anything new for vertex operator algebras since similar (in fact, stronger) results have been obtained in \cite{McR1, McR}. But we do obtain the superalgebra generalization of \cite[Theorem 4.13]{McR1} in Corollary \ref{cor:vosa_orbifold}. For a coset-type extension $U\otimes V\subseteq A$ where $U$ and $V$ are mutual commutants in $A$ and both $U$ and $A$ are strongly rational, one can use Theorem \ref{thm:voa_rat_thm_intro} to show that $V$ is strongly rational provided it is $C_2$-cofinite and its tensor category of modules is rigid. The detailed proof is deferred to \cite{McR3}, where with more effort we use the strong rationality of $A$ and $U$, combined with $C_2$-cofiniteness for $V$, to show that the category of $V$-modules is rigid. That is, in \cite{McR3} we use Theorem \ref{thm:voa_rat_thm_intro} to reduce the coset rationality problem to the problem of $C_2$-cofiniteness for the coset subalgebra.

The remaining contents of this paper are as follows. Sections \ref{sec:C_ss_to_CA_ss} and \ref{sec:CA_ss_to_C_ss} deal with algebra objects in abstract tensor categories. In Section \ref{sec:C_ss_to_CA_ss}, we prove a semisimplicity result for algebras in semisimple tensor categories corresponding to Theorem \ref{thm:main_voa_thm_intro}(1); it is the same result as \cite[Theorem 3.3]{KO}, but we provide a different proof. In Section \ref{sec:CA_ss_to_C_ss}, we prove the categorical generalization of Theorem \ref{thm:main_voa_thm_intro}(2), showing that a rigid braided tensor category $\cC$ is semisimple given the existence of a suitable commutative algebra in $\cC$ with semisimple local module category. In Section \ref{sec:VOA_app}, we explain how the abstract results for algebras in tensor categories apply to vertex operator algebras, and in particular we prove Theorems \ref{thm:main_voa_thm_intro} and \ref{thm:voa_rat_thm_intro}. We also discuss what Theorem \ref{thm:voa_rat_thm_intro} says about orbifold-type subalgebras and extensions, and whether some of the conditions in Theorem \ref{thm:voa_rat_thm_intro} may be vacuous or redundant. In Section \ref{sec:VOSA}, we generalize Theorems \ref{thm:main_voa_thm_intro} and \ref{thm:voa_rat_thm_intro} to the case that $A$ is a vertex operator superalgebra and discuss what these results say about orbifold-type subalgebras and extensions.

\vspace{2.5mm}

\noindent\textbf{Acknowledgements.} An early version of this work was presented at the Oberwolfach Workshop  \textit{Subfactors and Applications} organized by Dietmar Bisch, Terry Gannon, Vaughan Jones, and Yasuyuki Kawahigashi; see \cite{McR-ext-ab} for the accompanying extended abstract. I would also like to thank Shashank Kanade and Ling Chen for further opportunities to present this work, and I would like to thank Thomas Creutzig, Yi-Zhi Huang, and Jinwei Yang for comments and suggestions.

\section{Algebras in semisimple tensor categories}\label{sec:C_ss_to_CA_ss}

Before proving that certain categories are semisimple, we should clarify what we mean by a semisimple category. As usual in a category $\cC$, a morphism $f: W_1\rightarrow W_2$ is \textit{surjective} if for any pair of morphisms $g,h: W_2\rightarrow W_3$, $g\circ f=h\circ f$ implies $g=h$. Then we say that $\cC$ is \textit{semisimple} if any surjection $f: W_1\rightarrow W_2$ in $\cC$ splits, that is, there is a section $\sigma: W_2\rightarrow W_1$ such that $f\circ\sigma=\Id_{W_2}$. When $\cC$ is abelian, this means that the canonical morphism $\ker f\oplus\im\sigma\rightarrow W_1$ will be an isomorphism, and this in turn means that any subobject of any object $W_1$ of $\cC$ has a complement. (Specifically, given a subobject $i: \widetilde{W}_1\hookrightarrow W_1$, take $W_2=\coker i$ and $f: W_1\rightarrow W_2$ the cokernel morphism, so that $\im i = \ker f$, and thus $W_1\cong\widetilde{W}_1\oplus W_2$.) This definition of semisimplicity does not imply that every object of $\cC$ is a coproduct of finitely many simple objects, unless we assume in addition that all objects of $\cC$ have finite length.

\begin{rema}
 The definition of semisimple category that we are using also applies to categories that are not necessarily abelian. For example, the Axiom of Choice amounts to the assertion that the category of sets is semisimple.
\end{rema}

Let $\cC$ be an $\mathbb{F}$-linear tensor category, not necessarily symmetric or braided, where $\mathbb{F}$ is a field. This means that $\cC$ is an $\mathbb{F}$-linear abelian monoidal category such that the tensor product bifunctor induces $\mathbb{F}$-bilinear maps on morphisms. We use $\tens$ to denote the tensor product bifunctor on $\cC$, $\vac$ to denote the unit object, $l$ and $r$ to denote the left and right unit isomorphisms, and $\cA$ to denote the associativity isomorphisms. For convenience, we will also assume $\mathbb{F}=\mathrm{End}_\cC(\vac)$.
\begin{defi}
 A \textit{(unital, associative) algebra} in $\cC$ is an object $A$ equipped with a \textit{multiplication} morphism $\mu_A: A\tens A\rightarrow A$ and a \textit{unit} morphism $\iota_A:\vac\rightarrow A$ satisfying the following properties:
 \begin{enumerate}
  \item Unit: $\mu_A\circ(\iota_A\tens\Id_A)= l_A$ and $\mu_A\circ(\Id_A\tens\iota_A)=r_A$.
  \item Associativity: $\mu_A\circ(\Id_A\tens\mu_A)=\mu_A\circ(\mu_A\tens\Id_A)\circ\cA_{A,A,A}$.
 \end{enumerate}
\end{defi}
\begin{defi}
 Suppose $(A,\mu_A,\iota_A)$ is an algebra in $\cC$. A \textit{(left) $A$-module} is an object $X$ of $\cC$ equipped with a multiplication $\mu_X: A\tens X\rightarrow X$ satisfying the following properties:
  \begin{enumerate}
   \item Left unit: $\mu_X\circ(\iota_A\tens\Id_X)= l_X$ and
   \item Associativity: $\mu_X\circ(\Id_A\tens\mu_X)=\mu_X\circ(\mu_A\tens\Id_X)\circ\cA_{A,A,X}$.
  \end{enumerate}
  A morphism $f: X_1\rightarrow X_2$ of left $A$-modules is a morphism in $\cC$ such that
  \begin{equation*}
   f\circ\mu_{X_1}=\mu_{X_2}\circ(\Id_A\tens f).
  \end{equation*}
\end{defi}
\begin{rema}
 Categories of right $A$-modules and $A$-bimodules can be defined similarly.
\end{rema}

\begin{exam}
 If $(A,\mu_A,\iota_A)$ is an algebra in $\cC$, then $(A,\mu_A)$ is an $A$-module.
\end{exam}

We denote the category of $A$-modules by $\cC_A$. The goal of this section is to present an alternate proof of \cite[Theorem 3.3]{KO}, which asserts that $\cC_A$ is semisimple if $\cC$ is semisimple and $A$ satisfies a non-vanishing dimension condition.
\begin{exam}
 Suppose $\cC$ is the (semisimple) category of vector spaces over a field $\mathbb{F}$ of characteristic $p$ and $A=\mathbb{F}[G]$ is the group algebra of a finite group $G$. Then $\cC_A$ is the category of $G$-modules over $\mathbb{F}$, and Maschke's Theorem asserts that this category is semisimple provided $p\nmid \vert G\vert$, that is, $\dim_\mathbb{F} \mathbb{F}[G]\neq 0$. In other words, $\cC_A$ is semisimple if the index of the trivial group in $G$ is non-zero in $\mathbb{F}$.
\end{exam}

The main theorem of this section is a direct categorical generalization of Maschke's Theorem for finite groups, and the proof we will present directly generalizes standard proofs of Maschke's Theorem. But first we need a categorical version of the dimension of $A$ as an object of $\cC$, or of the index of $\vac$ as a subalgebra of $A$. For this, we will assume that $A$ is a rigid and self-dual object of $\cC$, though we do not need the entire category $\cC$ to be rigid. Observe, for example, that Maschke's Theorem applies to infinite-dimensional representations of a finite group, even though the category of all vector spaces is not rigid.
\begin{assum}\label{assum:main_assum_1}
 The algebra $(A,\mu_A,\iota_A)$ in $\cC$ satisfies the following conditions:
 \begin{enumerate}
  \item There is a morphism $\varepsilon_A: A\rightarrow\vac$ such that $\varepsilon_A\circ\iota_A=\Id_\vac$.
  \item As an object of $\cC$, $A$ is rigid and self-dual with evaluation $\varepsilon_A\circ\mu_A: A\tens A\rightarrow\vac$ and coevaluation $i_A:\vac\rightarrow A\tens A$. That is, the rigidity compositions
  \begin{equation*}
   A\xrightarrow{l_A^{-1}}\vac\tens A\xrightarrow{i_A\tens\Id_A} (A\tens A)\tens A\xrightarrow{\cA_{A,A,A}^{-1}} A\tens(A\tens A)\xrightarrow{\Id_A\tens(\varepsilon_A\circ\mu_A)} A\tens\vac\xrightarrow{r_A} A
  \end{equation*}
and
\begin{equation*}
 A\xrightarrow{r_A^{-1}} A\tens\vac\xrightarrow{\Id_A\tens i_A} A\tens(A\tens A)\xrightarrow{\cA_{A,A,A}} (A\tens A)\tens A\xrightarrow{(\varepsilon_A\circ\mu_A)\tens\Id_A} \vac\tens A\xrightarrow{l_A} A
\end{equation*}
both equal $\Id_A$.
\item $\mu_A\circ i_A=[A:\vac]\iota_A$ for some \textit{index} $[A:\vac]\in\mathbb{F}$.
 \end{enumerate}
\end{assum}
We could also call the index $[A:\vac]$ the dimension of $A$ in $\cC$, but we caution that if $A$ is a ribbon tensor category, then $[A:\vac]$ is not the categorical dimension $\dim_\cC A$ unless the ribbon structure $\delta_A: A\rightarrow A^{**}$ is the identity on $A=A^{**}$. We can now state the main theorem of this section:
\begin{theo}\label{thm:main_cat_thm_1}
 Under Assumption \ref{assum:main_assum_1}, if $\cC$ is semisimple and $[A:\vac]\neq 0$, then $\cC_A$ is semisimple.
\end{theo}

We will use standard graphical calculus in the proof for simplicity and clarity. We will need the following lemma; see \cite[Lemma 3.7]{McR} for a proof:
\begin{lemma}\label{rigidlike_lemma}
  The two morphisms $A\rightarrow A\tens A$ in $\cC$ given by the compositions
  \begin{equation*}
   A\xrightarrow{l_A^{-1}} \vac\tens A\xrightarrow{i_A\tens 1_A} (A\tens A)\tens A\xrightarrow{\cA_{A,A,A}^{-1}} A\tens(A\tens A)\xrightarrow{1_A\tens\mu_A} A\tens A
  \end{equation*}
and
\begin{equation*}
 A\xrightarrow{r_A^{-1}} A\tens\vac\xrightarrow{1_A\tens i_A} A\tens(A\tens A)\xrightarrow{\cA_{A,A,A}} (A\tens A)\tens A\xrightarrow{\mu_A\tens 1_A} A\tens A
\end{equation*}
are equal. Diagrammatically,
\begin{align*}
 \begin{matrix}
  \begin{tikzpicture}[scale = .9, baseline = {(current bounding box.center)}, line width=0.75pt]
   \draw (2,0) -- (2,1.5) .. controls (2,1.7) .. (1.7,1.8);
   \draw (1.5,2.2) -- (1.5,2.6);
   \draw (1.3,1.8) .. controls (1,1.7) .. (1,1.5) .. controls (1,.8) and (0,.8) .. (0,1.5) -- (0,2.6);
   \draw[dashed] (2,.4) .. controls (.5,.5) .. (.5,1);
   \node at (1.5,2) [draw,minimum width=20pt,minimum height=10pt,thick, fill=white] {$\mu_A$};
   \node at (2,-.25) {$A$};
   \node at (1.5,2.85) {$A$};
   \node at (0,2.85) {$A$};
   \node at (0,.9) {$A$};
   \node at (1,.9) {$A$};
  \end{tikzpicture}
 \end{matrix} =
 \begin{matrix}
  \begin{tikzpicture}[scale = .9, baseline = {(current bounding box.center)}, line width=0.75pt]
   \draw (0,0) -- (0,1.5) .. controls (0,1.7) .. (.3,1.8);
   \draw (.5,2.2) -- (.5,2.6);
   \draw (.7,1.8) .. controls (1,1.7) .. (1,1.5) .. controls (1,.8) and (2,.8) .. (2,1.5) -- (2,2.6);
   \draw[dashed] (0,.4) .. controls (1.5,.5) .. (1.5,1);
   \node at (.5,2) [draw,minimum width=20pt,minimum height=10pt,thick, fill=white] {$\mu_A$};
   \node at (0,-.25) {$A$};
   \node at (.5,2.85) {$A$};
   \node at (2,2.85) {$A$};
   \node at (1,.9) {$A$};
   \node at (2,.9) {$A$};
  \end{tikzpicture}
 \end{matrix} .
\end{align*}
 \end{lemma}
 \begin{rema}
  The proof of Lemma \ref{rigidlike_lemma}, as given in \cite{McR}, is the only place where we use Assumption \ref{assum:main_assum_1}(1) and the assumption on the evaluation in Assumption \ref{assum:main_assum_1}(2). The rest of the proof of Theorem \ref{thm:main_cat_thm_1} only requires the existence of $i_A$ and Assumption \ref{assum:main_assum_1}(3), as well as right exactness of the tensoring functor $A\boxtimes\bullet$ (which follows from the assumption that $A$ is a rigid object in $\cC$).
 \end{rema}

 Now we proceed with the proof of the theorem:
 \begin{proof}
  To prove $\cC_A$ is semisimple, we need to show that any surjection $f: X_1\rightarrow X_2$ in $\cC_A$ splits. We first need to verify that $f$ is still surjective in $\cC$: let $C$ be a cokernel of $f$ in $\cC$ with cokernel morphism $c: X_2\rightarrow C$. The proof of \cite[Theorem 2.9]{CKM1} shows that if the tensoring functor $A\tens\bullet$ is right exact, then there is a $\cC$-morphism $\mu_C: A\tens C\rightarrow C$ such that $(C,\mu_C)$ is an object of $\cC_A$ and $c$ is a morphism in $\cC_A$. In fact, $A\tens\bullet$ is right exact because we assume $A$ is rigid, so because $f$ is surjective in $\cC_A$ and $c\circ f=0$, we get $c=0$. This shows that the cokernel of $f$ in $\cC$ is $0$, and thus $f$ is surjective in $\cC$ as well.

Now  because $\cC$ is semisimple and $f$ is surjective in $\cC$, there is a $\cC$-morphism $\sigma: X_2\rightarrow X_1$ such that $f\circ\sigma =\Id_{X_2}$, but we need to ``average $\sigma$ over $A$'' to turn it into a morphism of $A$-modules. Moreover, ``averaging over $A$'' must involve creating two copies of $A$ with the coevaluation $i_A$, since the desired $\cC_A$-morphism should somehow involve the $A$-actions on both $X_1$ and $X_2$.  Indeed, take $S: X_2\rightarrow X_1$ to be the composition
  \begin{align}\label{eqn:S_def}
   X_2\xrightarrow{l_{X_2}^{-1}} \vac\tens X_2\xrightarrow{i_A\tens\Id_{X_2}} (A\tens A)\tens X_2\xrightarrow{\cA_{A,A,X_2}^{-1}} & A\tens(A\tens X_2)\nonumber\\
   &\xrightarrow{\Id_A\tens\mu_{X_2}} A\tens X_2\xrightarrow{\Id_A\tens\sigma} A\tens X_1\xrightarrow{\mu_{X_1}} X_1.
  \end{align}
Diagrammatically,
\begin{align}\label{diag:s_def}
S=
 \begin{matrix}
  \begin{tikzpicture}[scale = .9, baseline = {(current bounding box.center)}, line width=0.75pt]
   \draw (2,0) -- (2,1.5) .. controls (2,1.7) .. (1.7,1.8);
   \draw (1.5,2.2) -- (1.5,3.2);
   \draw (1.3,1.8) .. controls (1,1.7) .. (1,1.5) .. controls (1,.8) and (0,.8) .. (0,1.5) -- (0,3.2) .. controls (0, 3.5) .. (.55,3.8);
   \draw (1.5,3.2) .. controls (1.5,3.5) .. (.95, 3.8);
    \draw (.75,4.2) -- (.75,4.6);
   \draw[dashed] (2,.4) .. controls (.5,.5) .. (.5,1);
   \node at (1.5,2) [draw,minimum width=20pt,minimum height=10pt,thick, fill=white] {$\mu_{X_2}$};
   \node at (2,-.25) {$X_2$};
   \node at (1.5,3)[draw,minimum width=20pt,minimum height=10pt,thick, fill=white] {$\sigma$};
   \node at (.75,4)[draw,minimum width=20pt,minimum height=10pt,thick, fill=white] {$\mu_{X_1}$};
   \node at (0,.9) {$A$};
   \node at (1,.9) {$A$};
   \node at (.75, 4.8) {$X_1$};
  \end{tikzpicture}
 \end{matrix}
 \end{align}
We first show that $S$ is a morphism in $\cC_A$, that is, 
$$\mu_{X_1}\circ(\Id_A\tens S)=S\circ\mu_{X_2}.$$
Beginning with $\mu_{X_1}\circ(\Id_A\tens S)$, we use the associativity of $\mu_{X_1}$, Lemma \ref{rigidlike_lemma}, and the associativity of $\mu_{X_2}$ to calculate:
\begin{align}\label{diag:s_ac_hom}
 \begin{matrix}
  \begin{tikzpicture}[scale = .9, baseline = {(current bounding box.center)}, line width=0.75pt]
   \draw (2,0) -- (2,1.5) .. controls (2,1.7) .. (1.7,1.8);
   \draw (1.5,2.2) -- (1.5,3.2);
   \draw (1.3,1.8) .. controls (1,1.7) .. (1,1.5) .. controls (1,.8) and (0,.8) .. (0,1.5) -- (0,3.2) .. controls (0, 3.5) .. (.55,3.8);
   \draw (1.5,3.2) .. controls (1.5,3.5) .. (.95, 3.8);
    \draw (-.1,5.2) -- (-.1,5.6);
    \draw (-1,0) -- (-1,4.2) .. controls (-1,4.5) .. (-.3,4.8);
    \draw (.75,4.2) .. controls (.75,4.5) .. (.1,4.8);
   \draw[dashed] (2,.4) .. controls (.5,.5) .. (.5,1);
   \node at (1.5,2) [draw,minimum width=20pt,minimum height=10pt,thick, fill=white] {$\mu_{X_2}$};
   \node at (2,-.25) {$X_2$};
   \node at (1.5,3)[draw,minimum width=20pt,minimum height=10pt,thick, fill=white] {$\sigma$};
   \node at (.75,4)[draw,minimum width=20pt,minimum height=10pt,thick, fill=white] {$\mu_{X_1}$};
   \node at (0,.9) {$A$};
   \node at (1,.9) {$A$}; 
   \node at (-.1,5)[draw,minimum width=20pt,minimum height=10pt,thick, fill=white] {$\mu_{X_1}$};
   \node at (-1,-.25) {$A$};
   \node at (-.1, 5.8) {$X_1$};
  \end{tikzpicture}
 \end{matrix} = 
 \begin{matrix}
  \begin{tikzpicture}[scale = .9, baseline = {(current bounding box.center)}, line width=0.75pt]
   \draw (2,0) -- (2,1.5) .. controls (2,1.7) .. (1.7,1.8);
   \draw (1.5,2.2) -- (1.5,3.2);
   \draw (1.3,1.8) .. controls (1,1.7) .. (1,1.5) .. controls (1,.8) and (0,.8) .. (0,1.5) -- (0,3.2) .. controls (0, 3.5) .. (-.3,3.8);
   \draw (1.5,3.2) --(1.5,4) .. controls (1.5,4.5) .. (.7, 4.8);
    \draw (.5,5.2) -- (.5,5.6);
    \draw (-1,0) -- (-1,3.2) .. controls (-1,3.5) .. (-.7,3.8);
    \draw (-.5,4.2) .. controls (-.5,4.5) .. (.3,4.8);
   \draw[dashed] (-1,.4) .. controls (.5,.5) .. (.5,1);
   \node at (1.5,2) [draw,minimum width=20pt,minimum height=10pt,thick, fill=white] {$\mu_{X_2}$};
   \node at (2,-.25) {$X_2$};
   \node at (1.5,3)[draw,minimum width=20pt,minimum height=10pt,thick, fill=white] {$\sigma$};
   \node at (-.5,4)[draw,minimum width=20pt,minimum height=10pt,thick, fill=white] {$\mu_{A}$};
   \node at (0,.9) {$A$};
   \node at (1,.9) {$A$}; 
   \node at (.5,5)[draw,minimum width=20pt,minimum height=10pt,thick, fill=white] {$\mu_{X_1}$};
   \node at (-1,-.25) {$A$};
   \node at (.5, 5.8) {$X_1$};
  \end{tikzpicture}
 \end{matrix} = 
 \begin{matrix}
  \begin{tikzpicture}[scale = .9, baseline = {(current bounding box.center)}, line width=0.75pt]
   \draw (2,0) -- (2,1.5) .. controls (2,1.7) .. (1.7,1.8);
   \draw (1.5,2.2) .. controls (1.5,2.5) .. (2.05,2.8);
   \draw (1.3,1.8) .. controls (1,1.7) .. (1,1.5) .. controls (1,.8) and (0,.8) .. (0,1.5) -- (0,4.2) .. controls (0,4.5) .. (.925,4.8);
   \draw[dashed] (2,.4) .. controls (.5,.5) .. (.5,1);
   \draw (3,0) -- (3,2.2) .. controls (3,2.5) .. (2.45, 2.8);
   \draw (2.25,3.2) -- (2.25, 3.8);
   \draw (2.25,4.2) .. controls (2.25,4.5) .. (1.325, 4.8);
   \draw (1.125,5.2) -- (1.125, 5.6);
   \node at (1.5,2) [draw,minimum width=20pt,minimum height=10pt,thick, fill=white] {$\mu_A$};
    \node at (2.25,3) [draw,minimum width=20pt,minimum height=10pt,thick, fill=white] {$\mu_{X_2}$};
     \node at (2.25,4) [draw,minimum width=20pt,minimum height=10pt,thick, fill=white] {$\sigma$};
      \node at (1.125,5) [draw,minimum width=20pt,minimum height=10pt,thick, fill=white] {$\mu_{X_1}$};
   \node at (2,-.25) {$A$};
   \node at (0,.9) {$A$};
   \node at (1,.9) {$A$};
   \node at (3,-.25) {$X_2$};
   \node at (1.125, 5.8) {$X_1$};
  \end{tikzpicture}
 \end{matrix} =
 \begin{matrix}
  \begin{tikzpicture}[scale = .9, baseline = {(current bounding box.center)}, line width=0.75pt]
   \draw (2,0) -- (2,1.5) .. controls (2,1.7) .. (1.7,1.8);
   \draw (1.5,2.2) -- (1.5,3.2);
   \draw (1.3,1.8) .. controls (1,1.7) .. (1,1.5) .. controls (1,.8) and (0,.8) .. (0,1.5) -- (0,3.2) .. controls (0, 3.5) .. (.55,3.8);
   \draw (1.5,3.2) .. controls (1.5,3.5) .. (.95, 3.8);
    \draw (.75,4.2) -- (.75,4.6);
    \draw (1.5,-1) .. controls (1.5,-.7) .. (1.8,-.4);
    \draw (2.5,-1) .. controls (2.5,-.7) .. (2.2,-.4);
   \draw[dashed] (2,.4) .. controls (.5,.5) .. (.5,1);
   \node at (1.5,2) [draw,minimum width=20pt,minimum height=10pt,thick, fill=white] {$\mu_{X_2}$};
   \node at (2.5,-1.25) {$X_2$};
   \node at (1.5,-1.25) {$A$};
   \node at (1.5,3)[draw,minimum width=20pt,minimum height=10pt,thick, fill=white] {$\sigma$};
   \node at (.75,4)[draw,minimum width=20pt,minimum height=10pt,thick, fill=white] {$\mu_{X_1}$};
   \node at (0,.9) {$A$};
   \node at (1,.9) {$A$};
   \node at (.75, 4.8) {$X_1$};
      \node at (2,-.2)[draw,minimum width=20pt,minimum height=10pt,thick, fill=white] {$\mu_{X_2}$};
  \end{tikzpicture},
 \end{matrix}
\end{align}
which is indeed $S\circ\mu_{X_2}$. Next, we calculate $f\circ S$ using the fact that $f$ is a $\cC_A$-morphism, the identity $f\circ\sigma=\Id_{X_2}$, the associativity of $\mu_{X_2}$, and Assumption \ref{assum:main_assum_1}(3):
\begin{align}\label{diag:f_s=id}
   \begin{matrix}
  \begin{tikzpicture}[scale = .9, baseline = {(current bounding box.center)}, line width=0.75pt]
   \draw (2,0) -- (2,1.5) .. controls (2,1.7) .. (1.7,1.8);
   \draw (1.5,2.2) -- (1.5,3.2);
   \draw (1.3,1.8) .. controls (1,1.7) .. (1,1.5) .. controls (1,.8) and (0,.8) .. (0,1.5) -- (0,3.2) .. controls (0, 3.5) .. (.55,3.8);
   \draw (1.5,3.2) .. controls (1.5,3.5) .. (.95, 3.8);
    \draw (.75,5.2) -- (.75,5.6);
    \draw (.75,4.2) -- (.75,4.8);
   \draw[dashed] (2,.4) .. controls (.5,.5) .. (.5,1);
   \node at (1.5,2) [draw,minimum width=20pt,minimum height=10pt,thick, fill=white] {$\mu_{X_2}$};
   \node at (2,-.25) {$X_2$};
   \node at (1.5,3)[draw,minimum width=20pt,minimum height=10pt,thick, fill=white] {$\sigma$};
   \node at (.75,4)[draw,minimum width=20pt,minimum height=10pt,thick, fill=white] {$\mu_{X_1}$};
   \node at (.75,5)[draw,minimum width=20pt,minimum height=10pt,thick, fill=white] {$f$};
   \node at (0,.9) {$A$};
   \node at (1,.9) {$A$};
   \node at (.75, 5.8) {$X_2$};
  \end{tikzpicture}
 \end{matrix} =
 \begin{matrix}
  \begin{tikzpicture}[scale = .9, baseline = {(current bounding box.center)}, line width=0.75pt]
   \draw (2,0) -- (2,1.5) .. controls (2,1.7) .. (1.7,1.8);
   \draw (1.5,2.2) -- (1.5,3.2);
   \draw (1.3,1.8) .. controls (1,1.7) .. (1,1.5) .. controls (1,.8) and (0,.8) .. (0,1.5) -- (0,4.2) .. controls (0, 4.5) .. (.55,4.8);
   \draw (1.5,4.2) .. controls (1.5,4.5) .. (.95, 4.8);
    \draw (.75,5.2) -- (.75,5.6);
    \draw (1.5,3.2) -- (1.5,3.8);
   \draw[dashed] (2,.4) .. controls (.5,.5) .. (.5,1);
   \node at (1.5,2) [draw,minimum width=20pt,minimum height=10pt,thick, fill=white] {$\mu_{X_2}$};
   \node at (2,-.25) {$X_2$};
   \node at (1.5,3)[draw,minimum width=20pt,minimum height=10pt,thick, fill=white] {$\sigma$};
   \node at (.75,5)[draw,minimum width=20pt,minimum height=10pt,thick, fill=white] {$\mu_{X_2}$};
   \node at (1.5,4)[draw,minimum width=20pt,minimum height=10pt,thick, fill=white] {$f$};
   \node at (0,.9) {$A$};
   \node at (1,.9) {$A$};
   \node at (.75, 5.8) {$X_2$};
  \end{tikzpicture}
 \end{matrix} =
 \begin{matrix}
  \begin{tikzpicture}[scale = .9, baseline = {(current bounding box.center)}, line width=0.75pt]
   \draw (2,0) -- (2,1.5) .. controls (2,1.8) .. (1.7,2.1);
   \draw (1.3,2.1) .. controls (1,1.8) .. (1,1.5) .. controls (1,.8) and (0,.8) .. (0,1.5) -- (0,2.5) .. controls (0, 3.2) .. (.55,3.5);
   \draw (1.5,2.5) .. controls (1.5,3.2) .. (.95, 3.5);
    \draw (.75,3.9) -- (.75,4.5);
   \draw[dashed] (2,.4) .. controls (.5,.5) .. (.5,1);
   \node at (1.5,2.3) [draw,minimum width=20pt,minimum height=10pt,thick, fill=white] {$\mu_{X_2}$};
   \node at (2,-.25) {$X_2$};
   \node at (.75,3.7)[draw,minimum width=20pt,minimum height=10pt,thick, fill=white] {$\mu_{X_2}$};
   \node at (0,.9) {$A$};
   \node at (1,.9) {$A$};
   \node at (.75, 4.75) {$X_2$};
  \end{tikzpicture}
  \end{matrix} =
  \begin{matrix}
  \begin{tikzpicture}[scale = .9, baseline = {(current bounding box.center)}, line width=0.75pt]
   \draw (2,0) -- (2,2.5) .. controls (2,3.2) .. (1.45,3.5);
   \draw (.7,2.1) .. controls (1,1.8) .. (1,1.5) .. controls (1,.8) and (0,.8) .. (0,1.5) .. controls (0, 1.8) .. (.3,2.1);
   \draw (.5,2.5) .. controls (.5,3.2) .. (1.05, 3.5);
    \draw (1.25,3.9) -- (1.25,4.5);
   \draw[dashed] (2,.4) .. controls (.5,.5) .. (.5,1);
   \node at (.5,2.3) [draw,minimum width=20pt,minimum height=10pt,thick, fill=white] {$\mu_{A}$};
   \node at (2,-.25) {$X_2$};
   \node at (1.25,3.7)[draw,minimum width=20pt,minimum height=10pt,thick, fill=white] {$\mu_{X_2}$};
   \node at (0,.9) {$A$};
   \node at (1,.9) {$A$};
   \node at (1.25, 4.75) {$X_2$};
  \end{tikzpicture}
 \end{matrix} = [A:\vac]
 \begin{matrix}
  \begin{tikzpicture}[scale = .9, baseline = {(current bounding box.center)}, line width=0.75pt]
   \draw (1,0) -- (1,2.5) .. controls (1,3) .. (.7,3.3);
   \draw (0,2.2) -- (0,2.5) .. controls (0,3) .. (.3, 3.3);
    \draw (.5,3.7) -- (.5,4.5);
   \draw[dashed] (1,.8) .. controls (0,1) .. (0,1.8);
   \node at (0,2) [draw,minimum width=20pt,minimum height=10pt,thick, fill=white] {$\iota_A$};
   \node at (1,-.25) {$X_2$};
   \node at (.5,3.5)[draw,minimum width=20pt,minimum height=10pt,thick, fill=white] {$\mu_{X_2}$};
   \node at (.5, 4.75) {$X_2$};
  \end{tikzpicture}
 \end{matrix},
 \end{align}
which is $[A:\vac]\Id_{X_2}$ by the unit property of $\mu_{X_2}$. This shows that $s=\frac{1}{[A:\vac]}S$ is a $\cC_A$-morphism that satisfies $f\circ s=\Id_{X_2}$, completing the proof of the theorem.
\end{proof}

\begin{rema}
 Similar arguments show that if Assumption \ref{assum:main_assum_1} holds and $[A:\vac]\neq 0$, then the right $A$-module and $A$-bimodule categories also inherit semisimplicity from $\cC$. 
\end{rema}

\begin{exam}
 Let us examine the proof of the theorem in the case that $\cC$ is the category of $\mathbb{F}$-vector spaces and $A=\mathbb{F}[G]$ is the group algebra of a finite group. In this case, the coevaluation is given by
 \begin{equation*}
  i_A: 1\mapsto\sum_{g\in G} g\otimes g^{-1}
 \end{equation*}
and $[A:\vac]=\vert G\vert$. Thus if $f: X_1\rightarrow X_2$ is a $G$-module surjection and $\sigma: X_2\rightarrow X_1$ is a linear map such that $f\circ\sigma=\Id_{X_2}$, then the $G$-module homomorphism $S: X_2\rightarrow X_1$ of \eqref{eqn:S_def} is given by
\begin{align*}
 x_2\mapsto 1\otimes x_2\mapsto\sum_{g\in G} (g\otimes g^{-1})\otimes x_2\mapsto & \sum_{g\in G} g\otimes(g^{-1}\otimes x_2)\nonumber\\
 & \mapsto \sum_{g\in G} g\otimes\sigma(g^{-1}\cdot x_2) \mapsto\sum_{g\in G} g\cdot\sigma(g^{-1}\cdot x_2).
\end{align*}
That is, given a $G$-module surjection $f$, the $G$-module section $s: X_2\rightarrow X_1$ is obtained by averaging all conjugates  $g\circ\sigma\circ g^{-1}$ for $g\in G$:
\begin{equation*}
 s(x_2)=\frac{1}{\vert G\vert}\sum_{g\in G} g\cdot\sigma(g^{-1}\cdot x_2)
\end{equation*}
for $x_2\in X_2$. This averaging over $G$ is familiar from standard proofs of Maschke's Theorem.
\end{exam}

\begin{rema}\label{rema:AC_ss_to_C_ss}
See for example \cite[Lemma 4.21]{McR} for one converse to Theorem \ref{thm:main_cat_thm_1}: if $\cC_A$ is semisimple, then so is $\cC$ provided that the tensoring functor $A\tens\bullet$ preserves surjections and Assumption \ref{assum:main_assum_1}(1) holds. No assumption on the index $[A:\vac]$ is necessary. .
\end{rema}

\begin{rema}
 See \cite[Section 6]{KZ} for conditions under which a different converse to Theorem \ref{thm:main_cat_thm_1} holds: if $\cC$ and $\cC_A$ are both semisimple, then $[A:\vac]\neq 0$.
\end{rema}

\section{Semisimple algebras in braided tensor categories}\label{sec:CA_ss_to_C_ss}

In this section, we assume that the $\mathbb{F}$-linear tensor category $(\cC,\tens,\vac,l,r,\cA)$ of the previous section is braided, with natural braiding isomorphisms $\cR$.
\begin{defi}
 An algebra $(A,\mu_A,\iota_A)$ in $\cC$ is \textit{commutative} if $\mu_A\circ\cR_{A,A}=\mu_A$.
\end{defi}

If the tensoring functor $A\tens\bullet$ is right exact, then $\cC_A$ is a tensor category (see for example \cite[Section 1]{KO} or \cite[Section 2]{CKM1}). But we only get a braiding on the full subcategory $\cC_A^0$ of ``local'' $A$-modules, which are objects $(X,\mu_X)$ that satisfy
\begin{equation*}
 \mu_X\circ\cR_{X,A}\circ\cR_{A,X}=\mu_X.
\end{equation*}
So although Remark \ref{rema:AC_ss_to_C_ss} says that semisimplicity of $\cC_A$ implies semisimplicity of $\cC$ under mild conditions, a more interesting question for commutative algebras in braided tensor categories is whether semisimplicity of $\cC_A^0$ implies semisimplicity of $\cC$. The main result of this section is that the answer is yes if $\cC$ is rigid and the index $[A:\vac]$ is non-zero:
\begin{theo}\label{thm:main_cat_thm_2}
 Suppose $\cC$ is a rigid braided tensor category and $A$ is a commutative algebra in $\cC$ that satisfies Assumption \ref{assum:main_assum_1}. If $\cC_A^0$ is semisimple and $[A:\vac]\neq 0$, then $\cC$ is semisimple.
\end{theo}

To prove the theorem, we will use a projection functor $\Pi: \cC_A\rightarrow\cC_A^0$ that was used in \cite{KO, McR}; it is also a special case of the local induction functor introduced in \cite{FFRS} for not-necessarily-commutative symmetric special Frobenius algebras in ribbon categories. For an object $(X,\mu_X)$ in $\cC_A$, define $\pi_X\in\Endo_\cC(X)$ by the diagrammatic formula
\begin{equation*}
\pi_X=\frac{1}{[A:\vac]}
  \begin{matrix}
  \begin{tikzpicture}[scale = .9, baseline = {(current bounding box.center)}, line width=0.75pt]
   \node (w) at (2,-0.3) {$X$};
   \draw(1,2.25) .. controls (1,1.75) and (2,2) .. (2,1.5) -- (2,0);
   \draw[white, double=black, line width = 3pt ] (2,2.25) .. controls (2,1.75) and (1,2) .. (1,1.5) .. controls (1,0.8) and (0,0.8) .. (0,1.5) -- (0, 3.5) .. controls (0,4) .. (0.55,4.28);
   \draw[dashed] (2,0.4) .. controls (1,.4) .. (0.5,1);
   \draw (1.3, 3.3) .. controls (1,3.15) .. (1,3) .. controls (1,2.5) and (2,2.75) .. (2,2.25);
   \draw[white, double=black, line width = 3pt ] (1.7, 3.3) .. controls (2, 3.15) .. (2,3) .. controls (2,2.5) and (1,2.75) .. (1,2.25);
   \draw (1.5, 3.72) .. controls (1.5,4) .. (.95, 4.28);
   \draw (.75,4.72) -- (.75,5.25);
   \node (mu1) at (1.5, 3.5) [draw,minimum width=20pt,minimum height=10pt,thick, fill=white] {$\mu_X$};
   \node (mu2) at (.75, 4.5) [draw,minimum width=20pt,minimum height=10pt,thick, fill=white] {$\mu_X$};
   \node at (.75, 5.5) {$X$};
   \node at (-.1,.9) {$A$};
   \node at (1.1,.9) {$A$};
  \end{tikzpicture}
 \end{matrix}
\end{equation*}
Note that $\pi_X$ is defined somewhat differently in \cite{KO}, but the definitions are equivalent using the commutativity of $\mu_A$. By \cite[Lemma 4.3]{KO} or the $g=1$ case of \cite[Theorem 3.3]{McR}, $\pi_X$ is a morphism in $\cC_A$ and the image $\pi_X(X)$ is an object of $\cC_A^0$. 
\begin{lemma}\label{lem:piX_id_on_local}
 If $(X,\mu_X)$ is an object of $\cC_A^0$, then $\pi_X=\Id_X$.
\end{lemma}
\begin{proof}
 We calculate
 \begin{align*}
  \begin{matrix}
  \begin{tikzpicture}[scale = .9, baseline = {(current bounding box.center)}, line width=0.75pt]
   \node (w) at (2,-0.3) {$X$};
   \draw(1,2.25) .. controls (1,1.75) and (2,2) .. (2,1.5) -- (2,0);
   \draw[white, double=black, line width = 3pt ] (2,2.25) .. controls (2,1.75) and (1,2) .. (1,1.5) .. controls (1,0.8) and (0,0.8) .. (0,1.5) -- (0, 3.5) .. controls (0,4) .. (0.55,4.28);
   \draw[dashed] (2,0.4) .. controls (1,.4) .. (0.5,1);
   \draw (1.3, 3.3) .. controls (1,3.15) .. (1,3) .. controls (1,2.5) and (2,2.75) .. (2,2.25);
   \draw[white, double=black, line width = 3pt ] (1.7, 3.3) .. controls (2, 3.15) .. (2,3) .. controls (2,2.5) and (1,2.75) .. (1,2.25);
   \draw (1.5, 3.72) .. controls (1.5,4) .. (.95, 4.28);
   \draw (.75,4.72) -- (.75,5.25);
   \node (mu1) at (1.5, 3.5) [draw,minimum width=20pt,minimum height=10pt,thick, fill=white] {$\mu_X$};
   \node (mu2) at (.75, 4.5) [draw,minimum width=20pt,minimum height=10pt,thick, fill=white] {$\mu_X$};
   \node at (.75, 5.5) {$X$};
   \node at (-.1,.9) {$A$};
   \node at (1.1,.9) {$A$};
  \end{tikzpicture}
 \end{matrix} =
 \begin{matrix}
  \begin{tikzpicture}[scale = .9, baseline = {(current bounding box.center)}, line width=0.75pt]
   \node (w) at (2,-0.3) {$X$};
   \draw (2,0) -- (2,1.5) .. controls (2,1.8) .. (1.7,2.1);
   \draw[white, double=black, line width = 3pt ] (1,1.5) .. controls (1,0.8) and (0,0.8) .. (0,1.5) -- (0, 2.3) .. controls (0,2.8) .. (0.55,3.08);
   \draw (1,1.5) .. controls (1,1.8) .. (1.3,2.1);
   \draw[dashed] (2,0.4) .. controls (1,.4) .. (0.5,1);
   \draw (1.5, 2.52) .. controls (1.5,2.8) .. (.95, 3.08);
   \draw (.75,3.52) -- (.75,4.05);
   \node (mu1) at (1.5, 2.3) [draw,minimum width=20pt,minimum height=10pt,thick, fill=white] {$\mu_X$};
   \node (mu2) at (.75, 3.3) [draw,minimum width=20pt,minimum height=10pt,thick, fill=white] {$\mu_X$};
   \node at (.75, 4.3) {$X$};
   \node at (-.1,.9) {$A$};
   \node at (1.1,.9) {$A$};
  \end{tikzpicture}
 \end{matrix} =
 \begin{matrix}
  \begin{tikzpicture}[scale = .9, baseline = {(current bounding box.center)}, line width=0.75pt]
   \node (w) at (2,-0.3) {$X$};
   \draw (2,0) -- (2,2.5) .. controls (2,2.8) .. (1.45,3.1);
   \draw[white, double=black, line width = 3pt ] (1,1.5) .. controls (1,0.8) and (0,0.8) .. (0,1.5) .. controls (0,1.8) .. (.3,2.1);
   \draw (1,1.5) .. controls (1,1.8) .. (.7,2.1);
   \draw[dashed] (2,0.4) .. controls (1,.4) .. (0.5,1);
   \draw (.5, 2.5) .. controls (.5,2.8) .. (1.05, 3.1);
   \draw (1.25,3.52) -- (1.25,4.05);
   \node (mu1) at (.5, 2.3) [draw,minimum width=20pt,minimum height=10pt,thick, fill=white] {$\mu_A$};
   \node (mu2) at (1.25, 3.3) [draw,minimum width=20pt,minimum height=10pt,thick, fill=white] {$\mu_X$};
   \node at (1.25, 4.3) {$X$};
   \node at (-.1,.9) {$A$};
   \node at (1.1,.9) {$A$};
  \end{tikzpicture}
 \end{matrix} = [A:\vac]
 \begin{matrix}
  \begin{tikzpicture}[scale = .9, baseline = {(current bounding box.center)}, line width=0.75pt]
   \draw (1,.2) -- (1,2.5) .. controls (1,3) .. (.7,3.3);
   \draw (0,2.2) -- (0,2.5) .. controls (0,3) .. (.3, 3.3);
    \draw (.5,3.7) -- (.5,4.3);
   \draw[dashed] (1,.8) .. controls (0,1) .. (0,1.8);
   \node at (0,2) [draw,minimum width=20pt,minimum height=10pt,thick, fill=white] {$\iota_A$};
   \node at (1,-.05) {$X$};
   \node at (.5,3.5)[draw,minimum width=20pt,minimum height=10pt,thick, fill=white] {$\mu_{X}$};
   \node at (.5, 4.55) {$X$};
  \end{tikzpicture}
 \end{matrix} = [A:\vac]\Id_X
 \end{align*}
using the definition of $\cC_A^0$, the associativity of $\mu_X$, Assumption \ref{assum:main_assum_1}(3), and the left unit property of $\mu_X$.
\end{proof}
\begin{lemma}\label{lem:proj_natural}
 If $f: X_1\rightarrow X_2$ is a morphism in $\cC_A$, then $f\circ\pi_{X_1}=\pi_{X_2}\circ f$.
\end{lemma}
\begin{proof}
 This follows from the definition of morphisms in $\cC_A$ and the naturality of the braiding, associativity, and left unit isomorphisms in $\cC$:
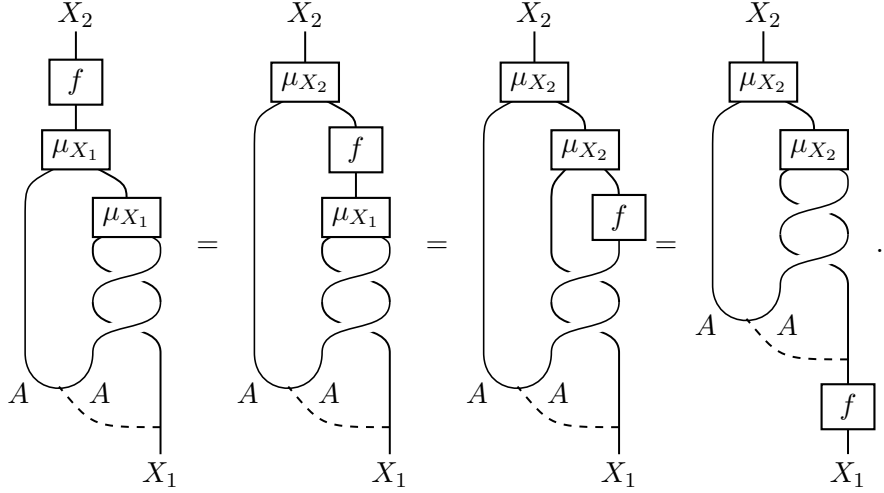
 \begin{align*}
  \begin{matrix}
  \begin{tikzpicture}[scale = .9, baseline = {(current bounding box.center)}, line width=0.75pt]
   \node (w) at (2,-0.3) {$X_1$};
   \draw(1,2.25) .. controls (1,1.75) and (2,2) .. (2,1.5) -- (2,0);
   \draw[white, double=black, line width = 3pt ] (2,2.25) .. controls (2,1.75) and (1,2) .. (1,1.5) .. controls (1,0.8) and (0,0.8) .. (0,1.5) -- (0, 3.5) .. controls (0,4) .. (0.55,4.28);
   \draw[dashed] (2,0.4) .. controls (1,.4) .. (0.5,1);
   \draw (1.3, 3.3) .. controls (1,3.15) .. (1,3) .. controls (1,2.5) and (2,2.75) .. (2,2.25);
   \draw[white, double=black, line width = 3pt ] (1.7, 3.3) .. controls (2, 3.15) .. (2,3) .. controls (2,2.5) and (1,2.75) .. (1,2.25);
   \draw (1.5, 3.72) .. controls (1.5,4) .. (.95, 4.28);
   \draw (.75,5.72) -- (.75,6.25);
   \draw (.75, 4.7) -- (.75, 5.3);
   \node (mu1) at (1.5, 3.5) [draw,minimum width=20pt,minimum height=10pt,thick, fill=white] {$\mu_{X_1}$};
   \node (mu2) at (.75, 4.5) [draw,minimum width=20pt,minimum height=10pt,thick, fill=white] {$\mu_{X_1}$};
   \node at (.75, 5.5) [draw,minimum width=20pt,minimum height=10pt,thick, fill=white] {$f$};
   \node at (.75, 6.5) {$X_2$};
   \node at (-.1,.9) {$A$};
   \node at (1.1,.9) {$A$};
  \end{tikzpicture}
 \end{matrix} = 
 \begin{matrix}
  \begin{tikzpicture}[scale = .9, baseline = {(current bounding box.center)}, line width=0.75pt]
   \node (w) at (2,-0.3) {$X_1$};
   \draw(1,2.25) .. controls (1,1.75) and (2,2) .. (2,1.5) -- (2,0);
   \draw[white, double=black, line width = 3pt ] (2,2.25) .. controls (2,1.75) and (1,2) .. (1,1.5) .. controls (1,0.8) and (0,0.8) .. (0,1.5) -- (0, 4.5) .. controls (0,5) .. (0.55,5.28);
   \draw[dashed] (2,0.4) .. controls (1,.4) .. (0.5,1);
   \draw (1.3, 3.3) .. controls (1,3.15) .. (1,3) .. controls (1,2.5) and (2,2.75) .. (2,2.25);
   \draw[white, double=black, line width = 3pt ] (1.7, 3.3) .. controls (2, 3.15) .. (2,3) .. controls (2,2.5) and (1,2.75) .. (1,2.25);
   \draw (1.5, 4.72) .. controls (1.5,5) .. (.95, 5.28);
   \draw (.75,5.72) -- (.75,6.25);
   \draw (1.5, 3.7) -- (1.5, 4.3);
   \node (mu1) at (1.5, 3.5) [draw,minimum width=20pt,minimum height=10pt,thick, fill=white] {$\mu_{X_1}$};
   \node (mu2) at (.75, 5.5) [draw,minimum width=20pt,minimum height=10pt,thick, fill=white] {$\mu_{X_2}$};
   \node at (1.5, 4.5) [draw,minimum width=20pt,minimum height=10pt,thick, fill=white] {$f$};
   \node at (.75, 6.5) {$X_2$};
   \node at (-.1,.9) {$A$};
   \node at (1.1,.9) {$A$};
  \end{tikzpicture}
 \end{matrix} =
 \begin{matrix}
  \begin{tikzpicture}[scale = .9, baseline = {(current bounding box.center)}, line width=0.75pt]
   \node (w) at (2,-0.3) {$X_1$};
   \draw(1,2.25) .. controls (1,1.75) and (2,2) .. (2,1.5) -- (2,0);
   \draw[white, double=black, line width = 3pt ] (2,2.25) .. controls (2,1.75) and (1,2) .. (1,1.5) .. controls (1,0.8) and (0,0.8) .. (0,1.5) -- (0, 4.5) .. controls (0,5) .. (0.55,5.28);
   \draw[dashed] (2,0.4) .. controls (1,.4) .. (0.5,1);
   \draw (1.3,4.3) .. controls (1,4) .. (1,3.7) -- (1,3) .. controls (1,2.5) and (2,2.75) .. (2,2.25);
   \draw[white, double=black, line width = 3pt ] (2, 3.3) -- (2,3) .. controls (2,2.5) and (1,2.75) .. (1,2.25);
   \draw (1.5, 4.72) .. controls (1.5,5) .. (.95, 5.28);
   \draw (.75,5.72) -- (.75,6.25);
   \draw (2, 3.7) .. controls (2,4) .. (1.7, 4.3);
   \node (mu1) at (1.5, 4.5) [draw,minimum width=20pt,minimum height=10pt,thick, fill=white] {$\mu_{X_2}$};
   \node (mu2) at (.75, 5.5) [draw,minimum width=20pt,minimum height=10pt,thick, fill=white] {$\mu_{X_2}$};
   \node at (2, 3.5) [draw,minimum width=20pt,minimum height=10pt,thick, fill=white] {$f$};
   \node at (.75, 6.5) {$X_2$};
   \node at (-.1,.9) {$A$};
   \node at (1.1,.9) {$A$};
  \end{tikzpicture}
 \end{matrix} =
 \begin{matrix}
  \begin{tikzpicture}[scale = .9, baseline = {(current bounding box.center)}, line width=0.75pt]
   \node (w) at (2,-1.3) {$X_1$};
   \draw(1,2.25) .. controls (1,1.75) and (2,2) .. (2,1.5) -- (2,0);
   \draw[white, double=black, line width = 3pt ] (2,2.25) .. controls (2,1.75) and (1,2) .. (1,1.5) .. controls (1,0.8) and (0,0.8) .. (0,1.5) -- (0, 3.5) .. controls (0,4) .. (0.55,4.28);
   \draw[dashed] (2,0.4) .. controls (1,.4) .. (0.5,1);
   \draw (1.3, 3.3) .. controls (1,3.15) .. (1,3) .. controls (1,2.5) and (2,2.75) .. (2,2.25);
   \draw[white, double=black, line width = 3pt ] (1.7, 3.3) .. controls (2, 3.15) .. (2,3) .. controls (2,2.5) and (1,2.75) .. (1,2.25);
   \draw (1.5, 3.72) .. controls (1.5,4) .. (.95, 4.28);
   \draw (.75,4.72) -- (.75,5.25);
   \draw (2,-1) -- (2,-.5);
   \node (mu1) at (1.5, 3.5) [draw,minimum width=20pt,minimum height=10pt,thick, fill=white] {$\mu_{X_2}$};
   \node (mu2) at (.75, 4.5) [draw,minimum width=20pt,minimum height=10pt,thick, fill=white] {$\mu_{X_2}$};
   \node at (2, -.3) [draw,minimum width=20pt,minimum height=10pt,thick, fill=white] {$f$};
   \node at (.75, 5.5) {$X_2$};
   \node at (-.1,.9) {$A$};
   \node at (1.1,.9) {$A$};
  \end{tikzpicture} 
 \end{matrix} .
 \end{align*}
\end{proof}

We now define the projection functor $\Pi: \cC_A\rightarrow\cC_A^0$ on objects by $\Pi(X)=\pi_X(X)$. We can factorize $\pi_X$ as
\begin{equation*}
 \pi_X: X\xrightarrow{\pi_X'}\Pi(X)\xrightarrow{u_X} X
\end{equation*}
with $\pi_X'$ surjective and $u_X$ the inclusion of the image into $X$. Then Lemma \ref{lem:proj_natural} shows that if $f: X_1\rightarrow X_2$ is a morphism in $\cC_A$, then restricting $f$ to $\Pi(X)$ defines a morphism $\Pi(f): \Pi(X_1)\rightarrow\Pi(X_2)$ such that both squares of the diagram
\begin{equation*}
 \xymatrixcolsep{4pc}
 \xymatrix{
 X_1 \ar[r]^{\pi_{X_1}'} \ar[d]^f & \Pi(X_1) \ar[d]^{\Pi(f)} \ar[r]^{u_{X_1}} & X_1 \ar[d]^{f} \\
 X_2 \ar[r]_{\pi_{X_2}'} & \Pi(X_2) \ar[r]_{u_{X_2}} & X_2\\
 }
\end{equation*}
commute. So $\Pi$ is a functor.

We will also need the induction functor $\cF: \cC\rightarrow\cC_A$ (see for example \cite[Theorem 1.6]{KO} or \cite[Section 2.7]{CKM1}). If $W$ is an object of $\cC$, then $\cF(W)=(A\tens W,\mu_{A\tens W})$ is an $A$-module where
\begin{equation*}
 \mu_{A\tens W}=(\mu_A\tens\Id_W)\circ\cA_{A,A,W}.
\end{equation*}
On morphisms, induction is defined by $\cF(f)=\Id_A\tens f$, and naturality of the right unit isomorphisms implies that $r_A: \cF(\vac)\rightarrow(A,\mu_A)$ is an isomorphism in $\cC_A$.

The projection and induction functors are used in the following lemma, which does not yet require the entire category $\cC$ to be rigid:
\begin{lemma}\label{lem:unit_proj}
 Let $\cC$ be a braided tensor category and $A$ a commutative algebra in $\cC$ that satisfies Assumption \ref{assum:main_assum_1}. If $\cC_A^0$ is semisimple and $[A:\vac]\neq 0$, then $\vac$ is projective in $\cC$.
\end{lemma}
\begin{proof}
 We need to show that if $f: W_1\rightarrow W_2$ is a surjection in $\cC$ and $g: \vac\rightarrow W_2$ is a morphism, then there is some $\sigma: \vac\rightarrow W_1$ such that $f\circ\sigma=g$.
 
 We start by applying the functor $\Pi\circ\cF: \cC\rightarrow\cC_A^0$ to $f$. We claim that $(\Pi\circ\cF)(f)$, that is, $\Pi(\Id_A\tens f)$,
 is a surjection in $\cC_A$, and thus also in $\cC_A^0$. In fact, $\Id_A\tens f$ is surjective in $\cC_A$ because the rigidity of $A$ from Assumption \ref{assum:main_assum_1}(2) means that $\cF$ is an exact functor. Then since $\pi_{A\tens W_2}'$ is also surjective, the commutative diagram
 \begin{equation*}
  \xymatrixcolsep{4pc}
  \xymatrix{
  A\tens W_1 \ar[d]_{\pi_{A\tens W_1}'} \ar[r]^{\Id_A\tens f} & A\tens W_2 \ar[d]^{\pi_{A\tens W_2}'}  \\
  \Pi(A\tens W_1) \ar[r]_{\Pi(\Id_A\tens f)} & \Pi(A\tens W_2)   \\
  }
 \end{equation*}
 implies that $\Pi(\Id_A\tens f)\circ\pi_{A\tens W_1}'$ is surjective. But then $\Pi(\Id_A\tens f)$ is surjective as well.

 Next, $\cF(\vac)\cong(A,\mu_A)$ is an object of $\cC_A^0$ by the commutativity of $\mu_A$. Since $\cC^0_A$ is semisimple, $A\tens\vac$ is projective in $\cC_A^0$ and we have a morphism
 \begin{equation*}
  s': A\tens\vac\rightarrow\Pi(A\tens W_1)
 \end{equation*}
in $\cC_A^0$ such that the diagram
\begin{equation*}
 \xymatrixcolsep{4pc}
 \xymatrix{
 A\tens \vac \ar[r]^{\Id_A\tens g} \ar[d]_{s'} & A\tens W_2 \ar[d]^{\pi_{A\tens W_2}'} \\
 \Pi(A\tens W_1) \ar[r]_{\Pi(\Id_A\tens f)} & \Pi(A\tens W_2) \\
 }
\end{equation*}
commutes. If we define $s=u_{A\tens W_1}\circ s'$, then the definitions say
\begin{align*}
 (\Id_A\tens f)\circ s & = u_{A\tens W_2}\circ\Pi(\Id_A\tens f)\circ s'\nonumber\\
  & =u_{A\tens W_2}\circ \pi_{A\tens W_2}'\circ(\Id_A\tens g)\nonumber\\
  & =\pi_{A\tens W_2}\circ(\Id_A\tens g).
\end{align*}
But because $A\tens\vac$ is an object of $\cC_A^0$, Lemmas \ref{lem:piX_id_on_local} and \ref{lem:proj_natural} imply that
\begin{equation*}
 \pi_{A\tens W_2}\circ(\Id_A\tens g)=(\Id_A\tens g)\circ\pi_{A\tens\vac}=\Id_A\tens g,
\end{equation*}
so actually $(\Id_A\tens f)\circ s=\Id_A\tens g$.

Finally, we can define $\sigma: \vac\rightarrow W_1$ in $\cC$ to be the composition
\begin{equation*}
 \vac\xrightarrow{l_\vac^{-1}} \vac\tens \vac\xrightarrow{\iota_A\tens\Id_\vac} A\tens\vac\xrightarrow{s} A\tens W_1\xrightarrow{\varepsilon_A\tens\Id_{W_1}} \vac\tens W_1\xrightarrow{l_{W_1}} W_1.
\end{equation*}
We calculate $f\circ\sigma$ using the commutative diagram
\begin{equation*}
\xymatrixcolsep{3pc}
 \xymatrix{
 \vac \ar[r]^{l_\vac^{-1}} \ar[rd]_{g} & \vac\tens \vac \ar[r]^{\iota_A\tens\Id_\vac} \ar[rd]_{\Id_\vac\tens g} & A\tens\vac \ar[r]^{s} \ar[rd]_{\Id_A\tens g} & A\tens W_1 \ar[r]^{\varepsilon_A\tens\Id_{W_1}} \ar[d]^{\Id_A\tens f} & \vac\tens W_1 \ar[r]^{l_{W_1}} \ar[d]^{\Id_\vac\tens f} & W_1 \ar[d]^{f} \\
 & W_2 \ar[r]_{l_{W_2}^{-1}} & \vac\tens W_2 \ar[r]_{\iota_A\tens\Id_{W_2}} & A\tens W_2 \ar[r]_{\varepsilon_A\tens\Id_{W_2}} & \vac\tens W_2 \ar[r]_{l_{W_2}} & W_2 \\
 }
\end{equation*}
Because  $\varepsilon_A\circ\iota_A=\Id_\vac$, we get $f\circ\sigma=g$, completing the proof that $\vac$ is projective in $\cC$.
\end{proof}

We will also use the following lemma which is essentially \cite[Corollary 4.2.13]{EGNO}, where, however, the assumptions are a bit stronger; see also the proof of \cite[Theorem 5.24]{CM}:
\begin{lemma}\label{lem:ss_from_1_proj}
 Suppose $\cC$ is a tensor category with projective unit object $\vac$. If $W_2$ is a rigid object of $\cC$ with rigid dual $W_2^*$, then any surjection $f: W_1\rightarrow W_2$ splits. 
\end{lemma}
\begin{proof}
 Given a surjection $f: W_1\rightarrow W_2$ in $\cC$, we wish to show that there is a section $\sigma: W_2\rightarrow W_1$ such that $f\circ\sigma=\Id_{W_2}$.
 
 Because $W_2^*$ is rigid, the tensoring functor $\bullet\tens W_2^*$ is exact and thus preserves surjections. So $f\tens\Id_{W_2^*}$ is still a surjection, and then because $\vac$ is projective, there is a morphism
 \begin{equation*}
  \widetilde{\sigma}: \vac\rightarrow W_1\tens W_2^*
 \end{equation*}
such that the diagram
\begin{equation*}
 \xymatrix{
  & W_1\tens W_2^* \ar[d]^{f\tens\Id_{W_2^*}} \\
  \vac \ar[ru]^(.4){\widetilde{\sigma}} \ar[r]_(.4){i_{W_2}} & W_2\tens W_2^* \\
 }
\end{equation*}
commutes. We set $\sigma$ equal to the composition
\begin{align*}
 W_2\xrightarrow{l_{W_2}^{-1}} \vac\tens W_2\xrightarrow{\widetilde{\sigma}\tens\Id_{W_2}} &(W_1\tens W_2^*)\tens W_2\nonumber\\
 &\xrightarrow{\cA_{W_1,W_2^*,W_2}^{-1}} W_1\tens(W_2^*\tens W_2)\xrightarrow{\Id_{W_1}\tens e_{W_2}} W_1\tens\vac\xrightarrow{r_{W_1}} W_1
\end{align*}
and calculate
\begin{align*}
 f\circ\sigma & = f\circ r_{W_1}\circ(\Id_{W_1}\tens e_{W_2})\circ\cA_{W_1,W_2^*,W_2}^{-1}\circ(\widetilde{\sigma}\tens\Id_{W_2})\circ l_{W_2}^{-1}\nonumber\\
 & = r_{W_2}\circ(f\tens\Id_\vac)\circ (\Id_{W_1}\tens e_{W_2})\circ\cA_{W_1,W_2^*,W_2}^{-1}\circ(\widetilde{\sigma}\tens\Id_{W_2})\circ l_{W_2}^{-1}\nonumber\\
 & =r_{W_2}\circ(\Id_{W_2}\tens e_{W_2})\circ\cA_{W_2,W_2^*,W_2}^{-1}\circ((f\tens\Id_{W_2^*})\tens\Id_{W_2})\circ(\widetilde{\sigma}\tens\Id_{W_2})\circ l_{W_2}^{-1}\nonumber\\
 & =r_{W_2}\circ(\Id_{W_2}\tens e_{W_2})\circ\cA_{W_2,W_2^*,W_2}^{-1}\circ(i_{W_2}\tens\Id_{W_2})\circ l_{W_2}^{-1}\nonumber\\
 & =\Id_{W_2}
\end{align*}
using naturality of the associativity and unit isomorphisms in $\cC$ and the rigidity of $W_2$.
\end{proof}

Now to finish the proof of Theorem \ref{thm:main_cat_thm_2}, we just combine Lemmas \ref{lem:unit_proj} and \ref{lem:ss_from_1_proj}:
\begin{proof}
In the setting of the theorem, the unit object $\vac$ of $\cC$ is projective by Lemma \ref{lem:unit_proj}. Then because every object of $\cC$ is rigid by assumption, Lemma \ref{lem:ss_from_1_proj} implies that every surjection in $\cC$ splits. Thus $\cC$ is semisimple.
\end{proof}

\section{Applications to vertex operator algebras}\label{sec:VOA_app}

We use the definition of vertex operator algebra from \cite{FLM, LL} and the definition of grading-restricted generalized module for a vertex operator algebra from, for example, \cite{Hu}, although we sometimes will simply use the term \textit{module} to refer to grading-restricted generalized modules. The theory of vertex and braided tensor categories of (grading-restricted generalized) modules for a vertex operator algebra is developed in \cite{HLZ1}-\cite{HLZ8}; see especially \cite{HLZ8} for a description of the braided tensor category structure, and see also the expositions \cite{HL} and \cite[Section 3.3]{CKM1}. However, we will not need the details of the vertex tensor category construction here, except to note that the unit object of such a category is the vertex operator algebra $V$ itself and that tensor products are determined by (logarithmic) intertwining operators among $V$-modules. Instead, we will need the dictionary due to Huang, Kirillov, and Lepowsky relating conformal inclusions of vertex operator algebras to commutative algebras in braided tensor categories:
\begin{theo}\label{thm:HKL_main_theo}\cite[Theorems 3.2 and 3.4]{HKL}
 Let $V\subseteq A$ be a conformal inclusion of vertex operator algebras and $\cC$ a category of grading-restricted generalized $V$-modules that includes $A$ and admits the braided tensor category structure of \cite{HLZ8}. Then:
 \begin{enumerate}
  \item The vertex operator algebra $A$ is a commutative algebra in $\cC$.
  \item The category $\cC^0_A$ of local $A$-modules is the full subcategory of grading-restricted generalized $A$-modules which are objects of $\cC$ when viewed as $V$-modules.
 \end{enumerate}
\end{theo}

The commutative algebra structure on the vertex operator algebra $A$ arises as follows: The unit $\iota_A: V\rightarrow A$ is the inclusion, and the multiplication $\mu_A: A\tens A\rightarrow A$ is the unique map induced by the vertex operator $Y_A: A\otimes A\rightarrow A((x))$, which is an intertwining operator among $V$-modules when $A$ is viewed as a $V$-module. Similarly, if $X$ is an $A$-module in the category $\cC$, the vertex operator $Y_X: A\otimes X\rightarrow X((x))$ induces the multiplication action $\mu_X: A\tens X\rightarrow X$. The module $X$ is in particular a local module for the commutative algebra $A$ because the vertex operator $Y_X$ involves only integral powers of the formal variable $x$.

Theorem \ref{thm:HKL_main_theo} means that we can use Theorems \ref{thm:main_cat_thm_1} and \ref{thm:main_cat_thm_2} to prove Theorem \ref{thm:main_voa_thm_intro} from the Introduction:
\begin{proof}
 The assumptions of Theorem \ref{thm:main_voa_thm_intro} are simply Assumption \ref{assum:main_assum_1}, so by Theorem \ref{thm:main_cat_thm_1}, the category $\cC_A$ of not-necessarily-local $A$-modules in $\cC$ is semisimple if $\cC$ is. Then the subcategory $\cC_A^0$ is also semisimple, provided that surjections in $\cC_A^0$ are still surjective in $\cC_A$. Thus suppose $f: X_1\rightarrow X_2$ is a surjection in $\cC_A^0$ and $g: X_2\rightarrow X_3$ is a morphism in $\cC_A$ such that $g\circ f=0$; we need to show $g=0$. One way is to factorize $g$ as 
 $$X_2\xrightarrow{g'}\im g\hookrightarrow X_3$$
 and note that $g'$ is a morphism in $\cC_A^0$ because $\cC_A^0$ is abelian (see for example \cite[Remark 2.57]{CKM1}). Thus $g'$ and hence $g$ also is zero. For an alternative argument, see \cite[Theorem 3.2(2)]{KO}. This proves the first assertion of Theorem \ref{thm:main_voa_thm_intro}.
 
 For the second assertion, the assumptions are identical to those of Theorem \ref{thm:main_cat_thm_2}, so we can immediately conclude $\cC$ is semisimple if $\cC$ is rigid and $\cC_A^0$ is semisimple.
\end{proof}

Before proving Theorem \ref{thm:voa_rat_thm_intro}, we recall in more detail what it means for a vertex operator algebra $V$ to be \textit{strongly rational}. This means:
\begin{itemize}
\item $V$ is positive energy, or CFT-type, that is, $V$ is $\NN$-graded by conformal weights, $V=\bigoplus_{n=0}^\infty V_{(n)}$, with $V_{(0)}$ equal to the linear span of the vacuum vector.
 \item $V$ is simple.
 \item $V$ is self-contragredient, that is, the contragredient module $V'=\bigoplus_{n=0}^\infty V_{(n)}^*$ is isomorphic to $V$ as a $V$-module. Equivalently, there is a nondegenerate invariant bilinear form $(\cdot,\cdot): V\times V\rightarrow\CC$.
 \item $V$ is $C_2$-cofinite, that is, $\dim V/C_2(V)<\infty$ where $C_2(V)=\mathrm{span}\lbrace u_{-2} v\,\vert\,u,v\in V\rbrace$.
 \item $V$ is rational, that is, every $\NN$-gradable weak $V$-module $W=\bigoplus_{n=0}^\infty W(n)$, where the $W(n)$ could be infinite dimensional, is the direct sum of irreducible (grading-restricted) $V$-modules.
\end{itemize}
The full category of grading-restricted modules for a strongly rational vertex operator algebra $V$ is a semisimple modular tensor category \cite{Hu-rig_mod}, while positive energy and $C_2$-cofiniteness are sufficient for braided tensor category structure on the full category of grading-restricted generalized modules \cite{Hu}.

Now we can prove Theorem \ref{thm:voa_rat_thm_intro} from the Introduction:
\begin{proof}
For the first assertion, we first assume $A$ is a simple, positive-energy vertex operator algebra in the modular tensor category $\cC$ of $V$-modules, and that $\dim_\cC A\neq 0$. Since $\cC$ is semisimple and $A$ has positive energy, $A=V\oplus\widetilde{A}$ as a $V$-module, where $\widetilde{A}\subseteq\bigoplus_{n\geq 1} A_{(n)}$. Then \cite[Corollary 3.2]{Li} implies $A$ is self-contragredient because $A$ is simple and
\begin{equation*}
 L(1)A_{(1)}=L(1)V_{(1)}\oplus L(1)\widetilde{A}_{(1)} =L(1)V_{(1)}=0,
\end{equation*}
where $L(1)V_{(1)}=0$ because $V$ is positive energy and self-contragredient. Also, $A$ is $C_2$-cofinite since by \cite[Proposition 5.2]{ABD}, $A$ is $C_2$-cofinite even as a $V$-module.

It remains to show that $A$ is rational. Since $A$ is positive energy and $C_2$-cofinite, it is enough to show that the category of grading-restricted generalized $A$-modules is semisimple (see Lemma 3.6 and Proposition 3.7 of \cite{CM} or \cite[Proposition 4.16]{McR1}). Since every grading-restricted generalized $A$-module is also a $V$-module, we need to show that $\cC_A^0$ is semisimple. So we just need to verify the assumptions needed for the first assertion of Theorem \ref{thm:main_cat_thm_1}. The $V$-homomorphism $\varepsilon_A$ exists because $A$ is a semisimple $V$-module, and because $A$ is simple, \cite[Lemma 1.20]{KO} shows that $A$ is self-dual with evaluation $\varepsilon_A\circ\mu_A$.

We also need to show that the non-zero categorical dimension $\dim_\cC A$ agrees with the index $[A:V]$. By definition,
 \begin{equation*}
  \dim_\cC A = \varepsilon_A\circ\mu_A\circ\cR_{A,A}\circ(\theta_A\tens\Id_A)\circ i_A,
 \end{equation*}
where $\cR_{A,A}$ is the braiding and $\theta_A$ is the twist $e^{2\pi i L(0)}$ (see \cite{Hu-rig_mod}). But $\theta_A=\Id_A$ because $A$ is $\ZZ$-graded by conformal weights, and $\mu_A\circ\cR_{A,A}=\mu_A$ because $A$ is a commutative algebra in $\cC$. Thus
\begin{equation*}
 \dim_\cC A = \varepsilon_A\circ\mu_A\circ i_A.
\end{equation*}
Now, $\mu_A\circ i_A$ is a $V$-module homomorphism from $V$ to $A$, which means it is determined by the image of the vacuum. Because $A$ is positive energy, such a homomorphism must be a multiple of the inclusion: $\mu_A\circ i_A =[A:V]\iota_A$ for some $[A:V]\in\CC$. Composing both sides of this equation on the left by $\varepsilon_A$ then shows that $[A:V]=\dim_\cC A\neq 0$. This completes the proof that $A$ is strongly rational.

Conversely, if $A$ is strongly rational, then $A$ is simple and positive energy by definition, and Theorems 6.10 and 6.3 of \cite{KZ} imply that the index $[A:V]$ (equivalently, $\dim_\cC A$) is non-zero. So the conditions of the first assertion of the theorem are necessary.

For the second assertion of the theorem, we first prove that $V$ is strongly rational assuming $A$ is strongly rational, $\varepsilon_A$ exists, $V$ is $C_2$-cofinite, the tensor category $\cC$ of grading-restricted generalized $V$-modules is rigid, and $\dim_\cC A\neq 0$. First we show that $V$ is simple: suppose $I\subseteq V$ is a non-zero ideal, equivalently a non-zero (left) $V$-submodule. Since $A=V\oplus\ker\varepsilon_A$ as a $V$-module, \cite[Proposition 4.5.6]{LL} implies that the $A$-submodule (equivalently, ideal) of $A$ generated by any $\tilde{v}\in I$ is
\begin{equation*}\label{eqn:v_tilde_ideal}
 \langle\tilde{v}\rangle =\text{span}\lbrace a_n\tilde{v}\,\vert\,a\in V\cup\ker\varepsilon_A, n\in\ZZ\rbrace.
\end{equation*}
If $a\in V$, then $a_n\tilde{v}\in I$, and if $a\in\ker\varepsilon_A$, then
\begin{equation*}
 Y_A(a,x)\tilde{v}= e^{xL(-1)}Y_A(\tilde{v},-x)a\in(\ker\varepsilon_A)((x))
\end{equation*}
by skew-symmetry for the vertex operator $Y_A$. Thus
\begin{equation*}
 \langle\tilde{v}\rangle\subseteq I\oplus\ker\varepsilon_A
\end{equation*}
for all $\tilde{v}\in I$; but if $\tilde{v}\neq 0$, then $\langle\tilde{v}\rangle=A$ since $A$ is simple. So in fact $I=V$ and $V$ is simple.

Positive energy for $V$ is immediate since $A$ has positive energy. To show that $V$ is self-contragredient, note that a nondegenerate invariant bilinear form $(\cdot,\cdot): A\times A\rightarrow\CC$ restricts to an invariant bilinear form on $V$ which is non-zero because it is nondegenerate on $V_{(0)}=A_{(0)}=\CC\vac$. But a non-zero invariant bilinear form on $V$ is nondegenerate because $V$ is simple, so $V$ is self-contragredient.
 
 It remains to show that $V$ is rational. As before, we need to show that the category $\cC$ of grading-restricted generalized $V$-modules is semisimple. Since we are assuming $\cC$ is a rigid braided tensor category, we just need to verify the assumptions for the second assertion of Theorem \ref{thm:main_voa_thm_intro}. In fact, $\varepsilon_A$ exists by assumption, and we can show that $A$ is self-dual with evaluation $\varepsilon_A\circ\mu_A$ and that $[A:V]\neq 0$ exactly as in the proof of the first assertion of the theorem. So $V$ is strongly rational.
 
 Conversely, if $V$ is strongly rational, then $\varepsilon_A$ exists because $A$ is a semisimple $V$-module, $V$ is $C_2$-cofinite by assumption, and the braided tensor category $\cC$ of $V$-modules is rigid \cite{Hu-rig_mod}. Also, $\dim_\cC A\neq 0$ by Theorems 6.10 and 6.3 of \cite{KZ} as before, so all the conditions in the second assertion of Theorem \ref{thm:voa_rat_thm_intro} are necessary.
\end{proof}

We apply Theorem \ref{thm:voa_rat_thm_intro} to the coset rationality problem in \cite{McR3}. For now, we discuss what the theorem says about the orbifold rationality problem:
\begin{exam}\label{exam:orbifold}
 Let $A$ be a strongly rational vertex operator algebra and $G$ a finite automorphism group, so that the fixed-point subalgebra $V=A^G$ is a conformal vertex operator subalgebra. By the main theorem of \cite{DLM}, $A$ is a semisimple $V$-module, so there is a $V$-module homomorphism $\varepsilon_A: A\rightarrow V$ such that $\varepsilon_A\circ\iota_A=\Id_V$. Moreover, if $\cC$ is a braided tensor category of $V$-modules that includes $A$, then $A$ is a rigid object in $\cC$ with $\dim_\cC A=\vert G\vert\neq 0$ by \cite[Proposition 4.15]{McR1}. Thus if we could show in addition:
 \begin{itemize}
  \item $V$ is $C_2$-cofinite, and
  \item The braided tensor category of grading-restricted generalized $V$-modules is rigid,
 \end{itemize}
then we could use Theorem \ref{thm:voa_rat_thm_intro}(2) to conclude that $V=A^G$ is strongly rational. But actually, we can do better than this: \cite[Corollary 4.23]{McR} shows that $C_2$-cofiniteness alone is enough to guarantee $A^G$ is strongly rational. In effect, the strong rationality of $A^G$ proved in \cite{CM} for $G$ cyclic, and in particular the rigidity of the category of $A^G$-modules in this special case, can be combined with \cite[Main Theorem 1]{McR} to yield the stronger result.

If conversely we assume that $V=A^G$ is strongly rational, and that $A$ is simple and positive energy, then we can use Theorem \ref{thm:voa_rat_thm_intro}(1) to conclude that $A$ is strongly rational, because $\dim_\cC A=\vert G\vert$ by \cite[Proposition 4.15]{McR1}. This was already observed in \cite[Theorem 4.14]{McR1}; in fact, the proof of Theorem \ref{thm:voa_rat_thm_intro}(1) here is just the straightforward adaptation of the proof  for the special case $A^G\subseteq A$ in \cite{McR1}.
\end{exam}

We conclude this section with a discussion of whether some of the conditions in Theorem \ref{thm:voa_rat_thm_intro} are vacuous or redundant. In Theorem \ref{thm:voa_rat_thm_intro}(1), the condition that $A$ be simple and positive energy is not vacuous because there exist vertex operator algebras which are neither simple nor positive energy but have non-zero dimension in the modular tensor category of a strongly rational subalgebra. Indeed, \cite[Proposition 2.10]{Li} shows that if $V$ is any vertex operator algebra (whether strongly rational or not) and $W$ is a $\ZZ$-graded $V$-module, then $V\oplus W$ has a vertex operator algebra structure given by
\begin{equation*}
 Y_{V\oplus W}((v_1,w_1),x)(v_2,w_2) =\left(Y_V(v_1,x)v_2, Y_W(v_1,x)w_2+e^{xL(-1)}Y_W(v_2,-x)w_1\right).
\end{equation*}
Taking $W=V$, we thus get a vertex operator algebra extension $A=V\oplus V$ for which the second copy of $V$ is a non-zero proper ideal; if $V$ is strongly rational, the dimension of $A$ as a $V$-module will of course be $2$.

As indicated in the Introduction, a more interesting question is whether the condition that $\dim_\cC A$ be non-zero is vacuous. More generally, does there exist a commutative algebra $A$ with trivial twist in a $\CC$-linear semisimple modular tensor category $\cC$ such that $\dim_\cC A =0$? It seems likely that such algebras exist, but note that they will never occur in unitary modular tensor categories, where dimensions of non-zero objects are strictly positive. Note also the relevance of the $\CC$-linearity assumption: over fields of positive characteristic, the group algebra of a finite abelian group can have dimension zero in the modular tensor category of finite-dimensional vector spaces.

Regarding the conditions of Theorem \ref{thm:voa_rat_thm_intro}(2), it is widely believed (see for example \cite{Hu-C2-conj}) that the tensor category of modules for a simple positive-energy self-contragredient $C_2$-cofinite vertex operator algebra should be rigid. So conjecturally the third condition in Theorem \ref{thm:voa_rat_thm_intro}(2) is redundant, but so far this is not proved. As for the remaining conditions, the following example shows that a rational vertex operator algebra $A$ can have a non-rational $C_2$-cofinite subalgebra with rigid module category $\cC$ when $\varepsilon_A$ fails to exist and/or $\dim_\cC A =0$:
\begin{exam}\label{exam:triplet}
 Consider the extension $\cW(p)\subseteq V_{\sqrt{2p}\mathbb{Z}}$, $p\in\mathbb{Z}_{\geq 2}$, where $\cW(p)$ is the triplet $W$-algebra and $V_{\sqrt{2p}\mathbb{Z}}$ is a lattice vertex operator algebra with modified conformal vector. The new Virasoro module structure on $V_{\sqrt{2p}\mathbb{Z}}$ is not semisimple: $\cW(p)$ is the maximal semisimple Virasoro submodule. Although $V_{\sqrt{2p}\mathbb{Z}}$ is rational, $\cW(p)$ is $C_2$-cofinite \cite{AM}, and the tensor category of grading-restricted generalized $\cW(p)$-modules is rigid \cite{TW}, $\cW(p)$ has a non-semisimple representation category and thus is not rational. Theorems \ref{thm:main_voa_thm_intro} and \ref{thm:voa_rat_thm_intro} fail because the inclusion $\cW(p)\hookrightarrow V_{\sqrt{2p}\mathbb{Z}}$ has no left inverse.

 Actually, $V_{\sqrt{2p}\mathbb{Z}}$ is not quite strongly rational because it is not self-contragredient with respect to the modified conformal structure, so strictly speaking, we are in the setting of Theorem \ref{thm:main_voa_thm_intro} rather than Theorem \ref{thm:voa_rat_thm_intro}. Since $V_{\sqrt{2p}\ZZ}$ is not self-dual in the rigid tensor category of $\cW(p)$-modules, we cannot define the index of the inclusion $\cW(p)\subseteq V_{\sqrt{2p}\ZZ}$ as in Theorem \ref{thm:main_voa_thm_intro}. However, the categorical dimension of $V_{\sqrt{2p}\ZZ}$ in the ribbon category of $\cW(p)$-modules still makes sense; we will show that this categorical dimension is $0$.
 
To determine $\dim_{\cW(p)} V_{\sqrt{2p}\mathbb{Z}}$, we use the (non-split) short exact $\cW(p)$-module sequence
 \begin{equation*}
  0\longrightarrow \cW(p)\longrightarrow V_{\sqrt{2p}\mathbb{Z}}\longrightarrow X_{p-1}^-\longrightarrow 0,
 \end{equation*}
where $X_{p-1}^-$ is a simple $\cW(p)$-module (using the notation of \cite{TW}). Its projective cover $P_{p-1}^-$, which was constructed explicitly in \cite{AM2} and shown to be projective in \cite{NT}, has length $4$ with two composition factors isomorphic to $X_{p-1}^-$ and two isomorphic to $\cW(p)$. Thus because categorical dimension is a well-defined function on the Grothendieck group of a ribbon category, we get
\begin{equation*}
 \dim_{\cW(p)} V_{\sqrt{2p}\ZZ} =\dim_{\cW(p)} \cW(p)+\dim_{\cW(p)} X_{p-1}^- =\frac{1}{2}\dim_{\cW(p)} P_{p-1}^- =0.
\end{equation*}
The last step uses the fact that the dimension of a projective object in a non-semisimple finite ribbon category is $0$, which is explained for example in \cite[Remark 4.6(1)]{GR}.
\end{exam}

\section{Generalization to vertex operator superalgebras}\label{sec:VOSA}

In this final section, we explain how Theorems \ref{thm:main_voa_thm_intro} and \ref{thm:voa_rat_thm_intro} generalize to the setting $V\subseteq A$ where $A$ is a vertex operator superalgebra and $V$ is a ($\ZZ$-graded) vertex operator algebra contained in the even part of $A$. 

There are slightly variant definitions of vertex operator superalgebra in the literature; see \cite[Section 2.2]{Li2} for a definition consistent with our usage of the term. In particular, a vertex operator superalgebra $A$ has two gradings, a $\frac{1}{2}\ZZ$-conformal weight grading $A=\bigoplus_{n\in\frac{1}{2}\ZZ} A_{(n)}$ and a $\ZZ/2\ZZ$-parity grading $A=A^{\bar{0}}\oplus A^{\bar{1}}$. While we do require these two gradings to be compatible in the sense that 
$$A_{(n)}=(A_{(n)}\cap A^{\bar{0}})\oplus(A_{(n)}\cap A^{\bar{1}})$$
for each $n\in\frac{1}{2}\ZZ$, we do not require that $A^{\bar{i}}=\bigoplus_{n\in\frac{i}{2}+\ZZ} A_{(n)}$ for $i=0,1$. In particular, this notion of vertex operator superalgebra incorporates the following four possibilities for $A$, using terminology from \cite{CKL}:
\begin{enumerate}
 \item Vertex operator algebra of correct statistics: $A$ is $\ZZ$-graded by conformal weights and $A^{\bar{1}}=0$,
 \item Vertex operator algebra of incorrect statistics: $A$ is $\frac{1}{2}\ZZ$-graded by conformal weights and $A^{\bar{1}}=0$.
 \item Vertex operator superalgebra of correct statistics: $A^{\bar{i}}=\bigoplus_{n\in\frac{i}{2}+\ZZ} A_{(n)}$ for $i=0,1$.
 \item Vertex operator superalgebra of incorrect statistics: $A$ is $\ZZ$-graded by conformal weights and $A^{\bar{i}}\neq 0$.
\end{enumerate}
Theorems \ref{thm:main_voa_thm_intro} and \ref{thm:voa_rat_thm_intro} cover case (1). Theorem \ref{thm:main_voa_thm_intro} remains valid in case (2) as well, but Theorem \ref{thm:voa_rat_thm_intro} requires modification since the index $[A:V]$ of Theorem \ref{thm:main_voa_thm_intro} is no longer the categorical dimension of $A$ as a $V$-module when $A$ is $\frac{1}{2}\ZZ$-graded.

For a general vertex operator superalgebra $A$, we use the notation $A_i=\bigoplus_{n\in\frac{i}{2}+\ZZ} A_{(n)}$ for $i=0,1$, so that $A_0\cap A^{\bar{0}}$ is a $\ZZ$-graded vertex operator algebra conformally embedded in $A$. We have two vertex operator (super)algebra automorphisms of $A$: the parity involution
\begin{equation*}
 P_A = \Id_{A^{\bar{0}}}\oplus(-\Id_{A^{\bar{1}}})
\end{equation*}
and the twist
\begin{equation*}
 \theta_A = \Id_{A_0}\oplus(-\Id_{A_1}).
\end{equation*}
Because the $\frac{1}{2}\ZZ$- and $\ZZ/2\ZZ$-gradings of $A$ are compatible, $\langle P_A,\theta_A\rangle$ is an abelian subgroup of $\mathrm{Aut}(A)$ of order at most $4$.

Let $A$ be a vertex operator superalgebra, $V$ a $\ZZ$-graded vertex operator algebra conformally embedded into $A_0\cap A^{(0)}$, and $\cC$ a category of grading-restricted generalized $V$-modules that includes $A$ and admits the braided tensor category structure of \cite{HLZ1}-\cite{HLZ8}. The analogue of Theorem \ref{thm:HKL_main_theo} for this setting was proved in \cite{CKL} and discussed further in \cite{CKM1}. To begin,  \cite[Section 2]{CKM1} constructs an auxiliary supercategory $\mathcal{SC}$ whose objects are $V$-modules in $\cC$ equipped with parity decompositions, and whose morphisms include both even and odd $V$-module homomorphisms. The \textit{underlying category} $\underline{\mathcal{SC}}$ has the same objects of $\mathcal{SC}$ but only the even morphisms, which preserve the parity decompositions of objects; $\underline{\mathcal{SC}}$ is a braided tensor category with braiding isomorphisms modified by a sign factor to account for parity. Then the following theorem comes from \cite[Theorems 3.13 and 3.14]{CKL} and \cite[Remark 2.27]{CKM1}:

\begin{theo}\label{thm:CKL_theo}
 Let $A$ be a vertex operator superalgebra, $V\subseteq A_0\cap A^{\bar{0}}$ a conformal inclusion of vertex operator algebras, and $\cC$ a category of grading-restricted generalized $V$-modules that includes $A$ and admits the braided tensor category structure of \cite{HLZ8}. Then:
 \begin{enumerate}
  \item The vertex operator superalgebra $A$ is a commutative algebra in $\underline{\mathcal{SC}}$. 
  \item The category $\underline{\mathcal{SC}}_A^0$ of local $A$-modules is the full subcategory of grading-restricted generalized $A$-modules which are objects of $\cC$ when viewed as $V$-modules.
 \end{enumerate}
\end{theo}

See for example \cite[Definition 3.1]{CKM1} for the definition of grading-restricted generalized module for a vertex operator superalgebra. We emphasize that morphisms in the braided tensor category  $\underline{\mathcal{SC}}_A^0$ are even, that is, they preserve parity decompositions of grading-restricted generalized $A$-modules. Thus $\underline{\mathcal{SC}}_A^0$ is semisimple in the sense of Section \ref{sec:C_ss_to_CA_ss} if and only if every $\ZZ/2\ZZ$-graded submodule of any object in $\underline{\mathcal{SC}}_A^0$ has a $\ZZ/2\ZZ$-graded complement; in particular, an object $X$ of $\underline{\mathcal{SC}}_A^0$ is simple if its only $\ZZ/2\ZZ$-graded submodules are $0$ and $X$. Note that a simple $A$-module in this sense may have non-trivial $A$-invariant subspaces that are not parity-graded; see \cite[Section 4.2.1]{CKM1} for a discussion of this issue.

Now the following is the superalgebra generalization of Theorem \ref{thm:main_voa_thm_intro}:
\begin{theo}\label{thm:main_vosa_thm}
 Let $A$ be a vertex operator superalgebra, $V\subseteq A_0\cap A^{\bar{0}}$ a conformal inclusion of vertex operator algebras, and $\cC$ a category of grading-restricted generalized $V$-modules that includes $A$ and admits the braided tensor category structure of \cite{HLZ8}. Assume that moreover:
 \begin{itemize}
  \item There is a $V$-homomorphism $\varepsilon_A: A\rightarrow V$ such that $\varepsilon_A\circ\iota_A=\Id_V$.
  \item The vertex operator superalgebra $A$ is a rigid and self-dual object of $\cC$ with evaluation $\varepsilon_A\circ\mu_A: A\tens A\rightarrow V$ and some coevaluation $i_A: V\rightarrow A\tens A$.
  \item For some non-zero index $[A:V]\in\CC$, $\mu_A\circ i_A=[A:V]\iota_A$.
 \end{itemize}
Then:
\begin{enumerate}
 \item If $\cC$ is semisimple, then $\underline{\mathcal{SC}}_A$ is also semisimple.
 \item If $\cC$ is rigid and $\underline{\mathcal{SC}}_A^0$ is semisimple, then $\cC$ is also semisimple.
\end{enumerate}
\end{theo}
\begin{proof}
 First note that $\cC$ is semisimple if and only if $\underline{\mathcal{SC}}$ is semisimple: In one direction, surjections in $\underline{\mathcal{SC}}$ have the form
 \begin{equation*}
  f^{\bar{0}}\oplus f^{\bar{1}}: W_1^{\bar{0}}\oplus W_1^{\bar{1}}\rightarrow W_2^{\bar{0}}\oplus W_2^{\bar{1}}
 \end{equation*}
where $f^{\bar{i}}: W_1^{\bar{i}}\rightarrow W_2^{\bar{i}}$ for $i=0,1$ are surjections in $\cC$. Thus surjections in $\underline{\mathcal{SC}}$ split if they do in $\cC$. Conversely, any surjection $f: W_1\rightarrow W_2$ in $\cC$ yields a surjection
\begin{equation*}
 f\oplus 0: W_1\oplus 0\rightarrow W_2\oplus 0
\end{equation*}
in $\underline{\mathcal{SC}}$, so $f$ splits if $\underline{\cS\cC}$ is semisimple.

Next, we claim that we may assume $\varepsilon_A$ is even, that is, $\varepsilon_A\vert_{A^{\bar{1}}}=0$, and thus $\varepsilon_A$ defines a morphism in $\underline{\mathcal{SC}}$. Indeed, if we let $p_{\bar{0}}: A\rightarrow A^{\bar{0}}$ denote the $V$-module projection with respect to the direct sum decomposition $A=A^{\bar{0}}\oplus A^{\bar{1}}$, then 
$$\varepsilon_A\circ p_{\bar{0}}\circ\iota_A=\varepsilon_A\circ\iota_A=\Id_V$$ 
since $V\subseteq A^{\bar{0}}$. So we may replace $\varepsilon_A$ with $\varepsilon_A\circ p_{\bar{0}}$ if necessary. 

Now because $\varepsilon_A$ and $\mu_A$ are both morphisms in $\underline{\mathcal{SC}}$, so is the evaluation $\varepsilon_A\circ\mu_A$. Then we may assume that the coevaluation $i_A: V\rightarrow A\tens A$ is even and thus defines a morphism in $\underline{\mathcal{SC}}$, as explained for example in the proof of \cite[Theorem 4.15]{McR}. Thus $A$ is rigid and self-dual in $\underline{\mathcal{SC}}$ as well as in $\cC$. Moreover, if $\cC$ is a rigid tensor category, then so is $\underline{\mathcal{SC}}$ (see \cite[Lemma 2.72]{CKM1}).

Replacing $i_A$ with its even part if necessary will not change the index $[A:V]$. Thus the two conclusions of the theorem follow immediately from Theorems \ref{thm:main_cat_thm_1} and \ref{thm:main_cat_thm_2}, since $A$ is a commutative algebra in $\underline{\mathcal{SC}}$ by Theorem \ref{thm:CKL_theo}. Note that in the first assertion we are using the full category $\underline{\mathcal{SC}}_A$ of non-local $A$-modules in $\underline{\mathcal{SC}}$.
\end{proof}

In the superalgebra setting, non-local $A$-modules associated to the group $\langle P_A,\theta_A\rangle\subseteq\mathrm{Aut}(A)$ are also of interest. For $g\in\langle P_A,\theta_A\rangle$, we say that an object $(X,\mu_X)$ in $\underline{\mathcal{SC}}_A$ is a \textit{$g$-twisted $A$-module} if 
\begin{equation*}
 \mu_X\circ\cR_{X,A}\circ\cR_{A,X}\circ(g\boxtimes\Id_X)=\mu_X;
\end{equation*}
in particular, a local module in $\underline{\mathcal{SC}}_A^0$ is the same as an \textit{untwisted} $A$-module in $\underline{\mathcal{SC}}$. When $A$ is a vertex operator (super)algebra of (in)correct statistics, the untwisted module category $\underline{\cS\cC}^0_A$ is called the ``Neveu-Schwarz sector'' in the physics literature; the category of $g$-twisted $A$-modules (where $g=P_A$ if $A$ is a superalgebra and $g=\theta_A$ if $A$ is an algebra of incorrect statistics) is called the ``Ramond sector.''

\begin{corol}\label{cor:vosa_tw_mods}
 In the setting and under the assumptions of Theorem \ref{thm:main_vosa_thm}, suppose that $\cC$ is semisimple. Then for each $g\in \langle P_A, \theta_A\rangle$, the category $\underline{\mathcal{SC}}^g_A$ of $g$-twisted $A$-modules is semisimple. In particular, $\underline{\mathcal{SC}}_A^0$ is semisimple.
\end{corol}
\begin{proof}
 As in the proof of Theorem \ref{thm:main_voa_thm_intro} in Section \ref{sec:VOA_app}, we just need to verify that any surjection $f: X_1\rightarrow X_2$ in $\underline{\mathcal{SC}}_A^g$ is still a surjection in $\underline{\mathcal{SC}}_A$. Thus suppose $h: X_2\rightarrow X_3$ is a morphism in $\underline{\mathcal{SC}}_A$ such that $h\circ f=0$; we need to show that $h=0$.
 
 It was shown in \cite[Section 3]{McR} that every object $X$ of $\underline{\mathcal{SC}}_A$ has an $\underline{\mathcal{SC}}_A$-endomorphism $\pi^g_X$, generalizing $\pi_X$ of Section \ref{sec:CA_ss_to_C_ss}, such that $\pi^g_X(X)$ is a $g$-twisted $A$-module, $\pi^g_X=\Id_X$ if $X$ is $g$-twisted, and $\pi^g_X$ commutes with morphisms in $\underline{\mathcal{SC}}_A$. In particular,
 \begin{equation*}
  h=h\circ\pi^g_{X_2} =\pi_{X_3}^g\circ h
 \end{equation*}
because $X_2$ is $g$-twisted, so $h$ actually defines a morphism $h: X_2\rightarrow\pi_{X_3}^g(X_3)$ in $\underline{\mathcal{SC}}_A^g$. Thus $h=0$ because $f$ is surjective in $\underline{\mathcal{SC}}_A^g$.
\end{proof}

Now we generalize Theorem \ref{thm:voa_rat_thm_intro} to the superalgebra setting. Note that rationality for $\frac{1}{2}\ZZ$-graded vertex operator (super)algebras is defined in terms of complete reducibility for $\frac{1}{2}\NN$-gradable weak modules.

\begin{theo}\label{thm:vosa_rat_1}
 Let $A$ be a vertex operator superalgebra and $V\subseteq A_0\cap A^{\bar{0}}$ a conformal inclusion of vertex operator algebras. 
 \begin{enumerate}
  \item If $V$ is strongly rational, then $A$ is also strongly rational provided:
 \begin{itemize}
  \item $A$ is simple and positive energy.
  \item The dimension of $A_0\cap A^{\bar{0}}$ in the modular tensor category of $V$-modules is non-zero.
 \end{itemize}
 
 \item If $A$ is strongly rational, then $V$ is also strongly rational provided:
 \begin{itemize}
  \item There is a $V$-module homomorphism $\varepsilon_A: A_0\cap A^{\bar{0}}\rightarrow V$ such that $\varepsilon_A\vert_V=\Id_V$.
  \item $V$ is $C_2$-cofinite.
  \item The tensor category $\cC$ of grading-restricted generalized $V$-modules is rigid.
  \item The categorical dimension $\dim_\cC A_0\cap A^{\bar{0}}$ is non-zero.
 \end{itemize}
 \end{enumerate}

\end{theo}
\begin{proof}
For the first assertion, $A$ is simple and positive energy by assumption, and this implies that $A$ is also self-contragredient as in the proof of Theorem \ref{thm:voa_rat_thm_intro}(1) (see also \cite[Remark 4.11]{CKM1}). Also, $A$ is $C_2$-cofinite as in the proof of Theorem \ref{thm:voa_rat_thm_intro}(1). It remains to show that $A$ is rational.

 Since $A$ is simple and positive energy, the same holds for $A_0\cap A^{\bar{0}}$ (simplicity follows from the main theorem of \cite{DLM}, or from \cite[Theorem 3.2]{McR1} which explicitly covers the superalgebra generality). Thus $A_0\cap A^{\bar{0}}$ is strongly rational by Theorem \ref{thm:voa_rat_thm_intro}(1), and we may consider the orbifold-type extension $A_0\cap A^{\bar{0}}=A^{\langle P_A, \theta_A\rangle}\subseteq A$. The assumptions of Theorem \ref{thm:main_vosa_thm} for such extensions were verified in the proof of \cite[Theorem 4.15]{McR}, so Theorem \ref{thm:main_vosa_thm}(1) and Corollary \ref{cor:vosa_tw_mods} show that every grading-restricted generalized $A$-module is semisimple.
 
 Now to show that $A$ is rational, let $X=\bigoplus_{n\in\frac{1}{2}\NN} X(n)$ be a $\frac{1}{2}\NN$-gradable weak $A$-module. Then $X$ restricts to an $\NN$-gradable weak $V$-module, so because $V$ is rational, $X\cong\bigoplus_{i\in I} W_i$ for irreducible $V$-modules $W_i$; we may assume that either $W_i\subseteq X^{\bar{0}}$ or $W_i\subseteq X^{\bar{1}}$ for each $i$ because $X^{\bar{0}}$ and $X^{\bar{1}}$ are $\NN$-gradable $V$-submodules of $X$. Let $X_i$ denote the $\ZZ/2\ZZ$-graded $A$-submodule of $X$ generated by $W_i$; by \cite[Proposition 4.5.6]{LL}, or \cite[Lemma 3.74]{CKM1}, it is the image of the $V$-module intertwining operator $Y_X\vert_{A\otimes W_i}$ of type $\binom{X}{A\,W_i}$. Then because $A$ and $W_i$ are $C_2$-cofinite $V$-modules, $X_i$ is also $C_2$-cofinite as a $V$-module by the Main Theorem of \cite{Mi1}. In particular, $X_i$ is a grading-restricted generalized $A$-submodule of $X$ and hence is semisimple.
 
 We have now shown that $X=\sum_{i\in I} X_i$ where the $X_i$ are semisimple grading-restricted generalized $A$-submodules. That is, $X$ is a sum, and hence also a direct sum, of irreducible grading-restricted $A$-modules. This shows that $A$ is rational, completing the proof of the first assertion of the theorem.
 
 For the second assertion, the assumptions on $V$ and its module category are enough for strong rationality of $A_0\cap A^{\bar{0}}$ to imply strongly rationality of $V$ via Theorem \ref{thm:voa_rat_thm_intro}. To show that $A_0\cap A^{\bar{0}}$ is in fact strongly rational, note that it is simple by the main theorem of \cite{DLM}, or by \cite[Theorem 3.2]{McR1}, because $A_0\cap A^{\bar{0}} = A^{\langle P_A,\theta_A\rangle}$. Then $A_0\cap A^{\bar{0}}$ is positive energy and self-contragredient because $A$ is, just as in the proof of Theorem \ref{thm:voa_rat_thm_intro}(2). Moreover, $A_0\cap A^{\bar{0}}$ is $C_2$-cofinite because it is a module for the $C_2$-cofinite vertex operator algebra $V$.

 To show $A_0\cap A^{\bar{0}}$ is rational and thus prove second assertion of the theorem, it is enough to show that the category of grading-restricted generalized $A_0\cap A^{\bar{0}}$-modules is semisimple; this is nothing but the category $\cC_{A_0\cap A^{\bar{0}}}^0$ of local modules for the commutative algebra $A_0\cap A^{\bar{0}}$ in the rigid braided tensor category $\cC$ of $V$-modules. To prove that $\cC_{A_0\cap A^{\bar{0}}}^0$ is semisimple, we want to apply Theorem \ref{thm:main_vosa_thm}(2) with $V$ replaced by $A_0\cap A^{\bar{0}}$ and with $\cC$ replaced by $\cC_{A_0\cap A^{\bar{0}}}^0$. As mentioned above, the three basic assumptions of Theorem \ref{thm:main_vosa_thm} were verified for orbifold-type extensions in the proof of \cite[Theorem 4.15]{McR}, but we also need to show that $\cC_{A_0\cap A^{\bar{0}}}^0$ is rigid. In fact, since $A_0\cap A^{\bar{0}}$ is a commutative algebra in $\cC$, the rigidity of $\cC_{A_0\cap A^{\bar{0}}}^0$ will follow from \cite[Theorem 1.15]{KO}  provided that the index $[A: A_0\cap A^{\bar{0}}]$ is non-zero. But just as in the proof of Theorem \ref{thm:voa_rat_thm_intro}, this is the same as $\dim_\cC A_0\cap A^{\bar{0}}$, which is non-zero by assumption.
\end{proof}

Theorem \ref{thm:vosa_rat_1} has the following specializations to orbifold-type extensions. The first
 generalizes \cite[Theorem 4.13]{McR1} and is also similar to \cite[Theorem 4.1]{DH}, which was proved for finite solvable automorphism groups of vertex operator superalgebras:
\begin{corol}\label{cor:vosa_orbifold}
 Suppose $A$ is a simple positive-energy vertex operator superalgebra and $G$ is a finite group of automorphisms of $A$ that includes $P_A$ and $\theta_A$. If $A^G$ is strongly rational, then $A$ is also strongly rational.
\end{corol}
\begin{proof}
 Since $G$ contains $P_A$ and $\theta_A$, the fixed-point subalgebra $A^G$ is a vertex operator subalgebra of $A_0\cap A^{\bar{0}}$. Thus to apply Theorem \ref{thm:vosa_rat_1}(1), we just need to verify that the dimension of $A_0\cap A^{\bar{0}}$ in the modular tensor category of $A^G$-modules is non-zero. In fact, since $A^G=(A_0\cap A^{\bar{0}})^{G/K}$, where $K$ is the kernel of the restriction homomorphism from $G$ to $\mathrm{Aut}(A_0\cap A^{\bar{0}})$, the dimension of $A_0\cap A^{\bar{0}}$ as an $A^G$-module is $\vert G/K\vert$ by \cite[Proposition 4.15]{McR1}.
\end{proof}

\begin{exam}
Similar to Example \ref{exam:orbifold}, let $A$ be a strongly rational vertex operator superalgebra and let $G$ be a finite group of automorphisms of $A$ that includes $P_A$ and $\theta_A$. If $A$ is strongly rational, then Theorem \ref{thm:vosa_rat_1}(2) shows that $V=A^G$ is also strongly rational provided that $V$ is $C_2$-cofinite and the tensor category $\cC$ of grading-restricted generalized $V$-modules is rigid. However, \cite[Corollary 4.23]{McR} shows that we can replace the assumptions that $A$ is strongly rational and that $\cC$ is rigid with the assumption that $A_0\cap A^{\bar{0}}$ is strongly rational (still retaining $C_2$-cofiniteness of $V$ as an assumption).
\end{exam}

\appendix

\section{Detailed calculations for Equation \texorpdfstring{\eqref{diag:s_ac_hom}}{(2.3)}}

Here we provide more details for the calculation \eqref{diag:s_ac_hom} in the proof of Theorem \ref{thm:main_cat_thm_1}, so that the interested reader may see the role of the pentagon and triangle axioms. These properties of a tensor category are not used in the other diagrammatic calculations of Sections \ref{sec:C_ss_to_CA_ss} and \ref{sec:CA_ss_to_C_ss}.

 We start with $\mu_{X_1}\circ(\Id_A\tens S)$, which is the composition
\begin{align*}
 A\tens & X_2\xrightarrow{\Id_A\tens l_{X_2}^{-1}} A\tens(\vac\tens X_2)\xrightarrow{\Id_A\tens(i_A\tens\Id_{X_2})} A\tens((A\tens A)\tens X_2)\nonumber\\
 &\xrightarrow{\Id_A\tens\cA_{A,A,X_2}^{-1}} A\tens(A\tens(A\tens X_2))\xrightarrow{\Id_A\tens(\Id_A\tens \mu_{X_2})} A\tens(A\tens X_2)\nonumber\\
 &\xrightarrow{\Id_A\tens(\Id_A\tens\sigma)} A\tens(A\tens X_1) \xrightarrow{\Id_A\tens\mu_{X_1}} A\tens X_1\xrightarrow{\mu_{X_1}} X_1.
\end{align*}
Using the triangle axiom, the associativity of $\mu_{X_1}$, and the naturality of the associativity isomorphisms, this becomes
\begin{align*}
 A\tens & X_2\xrightarrow{r_A^{-1}\tens\Id_{X_2}} (A\tens\vac)\tens X_2\xrightarrow{(\Id_A\tens i_A)\tens\Id_{X_2}} (A\tens(A\tens A))\tens X_2\nonumber\\
 & \xrightarrow{\cA_{A,A\tens A,X_2}^{-1}} A\tens((A\tens A)\tens X_2) \xrightarrow{\Id_A\tens\cA^{-1}_{A,A,X_2}} A\tens(A\tens (A\tens X_2))\nonumber\\
 & \xrightarrow{\cA_{A,A,A\tens X_2}} (A\tens A)\tens(A\tens X_2)\xrightarrow{\Id_{A\tens A}\tens\mu_{X_2}} (A\tens A)\tens X_2\nonumber\\
 & \xrightarrow{\Id_{A\tens A}\tens \sigma} (A\tens A)\tens X_1\xrightarrow{\mu_A\tens\Id_{X_1}} A\tens X_1\xrightarrow{\mu_{X_1}} X_1.
\end{align*}
We rewrite the composition of associativity isomorphisms here as $\cA_{A\tens A,A,X_2}^{-1}\circ(\cA_{A,A,A}\tens\Id_{X_2})$ using the pentagon axiom, and then move $\mu_A$ forward by the naturality of the associativity isomorphisms:
\begin{align*}
 A\tens & X_2\xrightarrow{r_A^{-1}\tens\Id_{X_2}} (A\tens\vac)\tens X_2\xrightarrow{(\Id_A\tens i_A)\tens\Id_{X_2}} (A\tens(A\tens A))\tens X_2\nonumber\\
 & \xrightarrow{\cA_{A,A,A}\tens\Id_{X_2}} ((A\tens A)\tens A)\tens X_2 \xrightarrow{(\mu_A\tens \Id_{A})\tens\Id_{X_2}} (A\tens A)\tens X_2\nonumber\\
 & \xrightarrow{\cA_{A,A,X_2}^{-1}} A\tens(A\tens X_2) \xrightarrow{\Id_A\tens\mu_{X_2}} A\tens X_2\xrightarrow{\Id_A\tens\sigma} A\tens X_1\xrightarrow{\mu_{X_1}} X_1.
\end{align*}
Now we apply Lemma \ref{rigidlike_lemma} to the first four arrows:
\begin{align*}
 A\tens & X_2 \xrightarrow{l_A^{-1}\tens\Id_{X_2}} (\vac\tens A)\tens X_2\xrightarrow{(i_A\tens\Id_A)\tens\Id_{X_2}} ((A\tens A)\tens A)\tens X_2\nonumber\\
 & \xrightarrow{\cA_{A,A,A}^{-1}\tens\Id_{X_2}} (A\tens(A\tens A))\tens X_2\xrightarrow{(\Id_A\tens\mu_A)\tens\Id_{X_2}} (A\tens A)\tens X_2\nonumber\\
 & \xrightarrow{\cA_{A,A,X_2}^{-1}} A\tens(A\tens X_2) \xrightarrow{\Id_A\tens\mu_{X_2}} A\tens X_2\xrightarrow{\Id_A\tens\sigma} A\tens X_1\xrightarrow{\mu_{X_1}} X_1.
\end{align*}
We rewrite the first two arrows using properties of the unit and naturality of the associativity isomorphisms; we also apply naturality of the associativity isomorphisms and the associativity of $\mu_{X_2}$ to get
\begin{align*}
 A\tens & X_2  \xrightarrow{l_{A\tens X_2}^{-1}}  \vac\tens(A\tens X_2)\xrightarrow{i_A\tens\Id_{A\tens X_2}} (A\tens A)\tens(A\tens X_2)\nonumber\\
 &\xrightarrow{\cA_{A\tens A,A,X_2}} ((A\tens A)\tens A)\tens X_2 \xrightarrow{\cA_{A,A,A}^{-1}\tens\Id_{X_2}} (A\tens(A\tens A))\tens X_2\nonumber\\
 &\xrightarrow{\cA^{-1}_{A,A\tens A,X_2}} A\tens((A\tens A)\tens X_2)\xrightarrow{\Id_A\tens\cA^{-1}_{A,A,X_2}} A\tens(A\tens(A\tens X_2))\nonumber\\
 &\xrightarrow{\Id_A\tens(\Id_A\tens\mu_{X_2})} A\tens(A\tens X_2)\xrightarrow{\Id_A\tens\mu_{X_2}} A\tens X_2\xrightarrow{\Id_A\tens\sigma} A\tens X_1\xrightarrow{\mu_{X_1}} X_1.
\end{align*}
By the pentagon axiom, the four associativity isomorphisms here simplify to $\cA_{A,A,A\tens X_2}^{-1}$, and then we move the first $\mu_{X_2}$ forward in the composition using the naturality of the associativity and left unit isomorphisms to get
\begin{align*}
 A\tens X_2\xrightarrow{\mu_{X_2}} X_2\xrightarrow{l_{X_2}^{-1}} & \vac\tens  X_2  \xrightarrow{i_A\tens\Id_{X_2}}  (A\tens A)\tens X_2\nonumber\\
 &\xrightarrow{\cA_{A,A,X_2}^{-1}} A\tens(A\tens X_2)\xrightarrow{\Id_A\tens\mu_{X_2}} A\tens X_2\xrightarrow{\Id_A\tens\sigma} A\tens X_1\xrightarrow{\mu_{X_1}} X_1.
\end{align*}
This is $S\circ\mu_{X_2}$, as required.

%
%
%

\end{document}